\documentclass[10pt,a4paper, reqno]{amsart}

\usepackage{geometry}
\usepackage[T1]{fontenc}
\usepackage{lmodern}
\usepackage{fix-cm}
\usepackage[right]{eurosym} 
\usepackage[utf8]{inputenc} 
\usepackage[ngerman, english]{babel}

\usepackage[mathscr]{euscript}
\usepackage{amsmath}
\usepackage{wasysym}
\usepackage{mathtools}
\usepackage{amsfonts}
\usepackage{amssymb}
\usepackage{bbm}
\usepackage{amscd}
\usepackage{esint}

\usepackage{graphicx}
\usepackage{float}
\usepackage[caption = false]{subfig}

\usepackage{color}
\usepackage{enumerate} 
\usepackage{hyperref} 
\usepackage{enumitem} 
\makeatletter
\def\namedlabel#1#2{\begingroup
	#2%
	\def\@currentlabel{#2}%
	\phantomsection\label{#1}\endgroup
}
\makeatother

\usepackage{verbatim} 

\usepackage{pgfplots}
\usepackage{tikz-cd}
\usetikzlibrary{quotes,babel,angles}
\usepgfplotslibrary{polar}
\usepgflibrary{shapes.geometric}
\usetikzlibrary{calc}
\usetikzlibrary{patterns}
\pgfdeclareplotmark{mystar}{
	\node[star,star point ratio=2.25,minimum size=6.5pt,
	inner sep=0pt,draw=black,solid,fill=black] {};
}
\pgfplotsset{compat=1.16}
\pgfplotsset{plot1/.append style={axis x line=middle, axis y line=
		middle, xlabel={$s$}, ylabel={$l$}, axis equal }}
\usetikzlibrary {patterns.meta}
\tikzdeclarepattern{
	name=mylines,
	parameters={
		\pgfkeysvalueof{/pgf/pattern keys/size},
		\pgfkeysvalueof{/pgf/pattern keys/angle},
		\pgfkeysvalueof{/pgf/pattern keys/line width},
	},
	bounding box={
		(0,-0.5*\pgfkeysvalueof{/pgf/pattern keys/line width}) and
		(\pgfkeysvalueof{/pgf/pattern keys/size},
		0.5*\pgfkeysvalueof{/pgf/pattern keys/line width})},
	tile size={(\pgfkeysvalueof{/pgf/pattern keys/size},
		\pgfkeysvalueof{/pgf/pattern keys/size})},
	tile transformation={rotate=\pgfkeysvalueof{/pgf/pattern keys/angle}},
	defaults={
		size/.initial=4pt,
		angle/.initial=130,
		line width/.initial=.5pt,
	},
	code={
		\draw [line width=\pgfkeysvalueof{/pgf/pattern keys/line width}]
		(0,0) -- (\pgfkeysvalueof{/pgf/pattern keys/size},0);
	},
}
\usepackage[capitalize,nameinlink,noabbrev]{cleveref} 

\newtheorem{theorem}{Theorem}[section]

\newtheorem{lemma}[theorem]{Lemma}
\newtheorem{lem}[theorem]{Lemma}
\newtheorem{proposition}[theorem]{Proposition}

\newtheorem{rem}[theorem]{Remark}

\newcommand{\BC}{\mathbb{C}}

\newcommand{\BN}{\mathbb{N}}

\newcommand{\BR}{\mathbb{R}}

\newcommand{\BZ}{\mathbb{Z}}

\newcommand{\SF}{\mathscr{F}}

\newcommand{\1}{\mathbbm{1}} 
\newcommand{\supp}{\text{ supp }}

\newcommand{\norm}[1]{\left\lVert#1\right\rVert}
\newcommand{\normtwo}[1]{{\left\vert\kern-0.25ex\left\vert\kern-0.25ex\left\vert #1 
		\right\vert\kern-0.25ex\right\vert\kern-0.25ex\right\vert}}

\DeclareMathOperator{\di}{d}

\newcommand{\abs}[1]{\ensuremath{\left\vert#1\right\vert}}

\newcommand{\lap}{\bigtriangleup}

\newcommand{\wo}{\langle \nabla \rangle}
\newcommand{\wop}[1]{\langle #1 \rangle }

\newcommand{\icol}[1]{
	\left(\begin{smallmatrix}#1\end{smallmatrix}\right)%
}

\usepackage{csquotes}
\usepackage[
doi=false,isbn=false,url=false,
backend=biber,
style=numeric,
sorting=nyvt
]{biblatex}
\usepackage{varwidth}

\usepackage{changepage}                 
\usepackage{lipsum}

	{\end{adjustwidth}}
	{\end{adjustwidth}}

 \usepackage{amsaddr}

\begin{document}
	\title{Norm inflation for the Zakharov system}
	\author{Florian Grube}
	\address{Fakultät für Mathematik, Universität Bielefeld, \\Postfach 10 01 31, 33501 Bielefeld, Germany}
	
	\makeatletter
	\@namedef{subjclassname@2020}{%
		\textup{2020} Mathematics Subject Classification}
	\makeatother
	
	
	
	\begin{abstract} We prove norm inflation in new regions of Sobolev regularities for the scalar Zakharov system in the spatial domain $\BR^d$ for arbitrary $d\in \BN$. To this end, we apply abstract considerations of Bejenaru and Tao from \cite{Taoideaillposedness} and modify arguments of Iwabuchi and Ogawa \cite{Iwabuchirefinedideaoftao}. This proves several results on well-posedness, which includes existence of solutions, uniqueness and continuous dependence on the initial data, to be sharp up to endpoints. 
	\end{abstract}

	\maketitle
	
	\section{Introduction}
	We consider the scalar Zakharov system
	\begin{align}\label{zakharovsystem}
	\begin{cases}
	i \partial_t u + \lap u &= v u \\
	\partial_t^2 v -\lap v &= \lap \abs{u}^2
	\end{cases}
	\end{align}
	with initial data $(u,v, \partial_t v)|_{t=0}= (u_0, v_0, v_1) \in H^s(\BR^d) \times H^l(\BR^d) \times H^{l-1}(\BR^d).$\medskip
	
	Zakharov introduced this model 1972 in his work \cite{originalzakharov} to describe Langmuir waves in weakly inducting plasma. In the notation above the real valued function $v: \BR\times \BR^d \to \BR$ models the ion density fluctuation and $u: \BR \times \BR^d \to \BC$ the complex envelope of the electric field.
	
	The aim of this paper is to prove norm inflation for \eqref{zakharovsystem} in some product Sobolev spaces to complement the rich local and global well-posedness theory for the Zakharov system. We say norm inflation, in alignment with the definition given in \cite{Kishimoto2019}, in $H^{s}(\BR^d)\times H^l(\BR^d)$ occurs if for any $\delta>0$ there exists initial data $(u_0,v_0, v_1)\in H^\infty(\BR^d)\times H^\infty(\BR^d)\times H^\infty(\BR^d)$ and a time $T>0$ such that 
	\begin{equation*}
		\norm{u_0}_{H^s}+ \norm{v_0}_{H^l}+\norm{v_1}_{H^{l-1}}<\delta,\quad T<\delta
	\end{equation*}
	and the corresponding smooth solution $(u,v)$ to \eqref{zakharovsystem} exists on $[0,T]$ and satisfies 
	\begin{equation*}
		\norm{u(T)}_{H^s}+\norm{v(T)}_{H^l}>\delta^{-1}.
	\end{equation*}
	This clearly implies the discontinuity of the solution map $(u_0,v_0, v_1)\mapsto (u,v)$ at the origin. 
	\subsection{Known local well-posedness results.} The local and global well-posedness theory for the Zakharov system has been extensively studied \cite{Ozawa1992, kenig1995, bourgaincolliander3d, ginibre, pecher01, colliander2008, pecher08, bejenaru2d, pecher09, Herrconvolution, zakharovindim4Herr, candy2019zakharov, Candy2021zakharov, chen2022, sanwal2021}. We limit the presentation of the past results to the currently best local well-posedness results available. For dimensions $d=1,2,3$, Sanwal proved in \cite{sanwal2021} that the Zakharov system is locally well-posed for initial data in $H^s(\BR^d)\times H^l(\BR^d)\times H^{l-1}(\BR^d)$ for 
	\begin{equation*}
		l>-\tfrac{1}{2}, \qquad \max\{ l-1, \tfrac{l}{2}+\tfrac{1}{4} \}<s<l+2
	\end{equation*}
	by modifying Fourier restriction spaces $X^{s,b}$, introducing new temporal weights and applying a contraction mapping principle. Partial endpoint cases were established in dimension $d=1$ by Ginibre, Tsutsumi and Velo in \cite{ginibre} showing that the lines $l=-\tfrac{1}{2}$ for $0\le s\le \tfrac{1}{2}$ and $2\,s=l+\tfrac{1}{2}$ for $0\le s<1$ are locally well-posed. By modifying the Fourier restriction spaces from \cite{ginibre}, Bejenaru, Herr, Holmer and Tataru proved in \cite{bejenaru2d} that the Zakharov system in two dimensions is locally well-posed in $L^2\times H^{-1/2}\times H^{-3/2}$ and the time of existence depends only on the norm of the initial data. This stands in contrast to the focusing, cubic Schrödinger equation, which arises as a subsonic limit from the Zakharov system, where this regularity is critical and the maximal time of existence depends directly on the initial data. Chen and Wu added the boundary line $l=s+1$ for $s\ge 2$ and $d=2,3$ to the locally well-posed regime in their work \cite{chen2022}. In \cite{candy2019zakharov} Candy, Herr and Nakanishi gave a complete answer to the question of local well-posedness in the energy critical and super critical dimensions $d\ge 4$. They proved local well-posedness of the Zakharov system with initial data in $H^s\times H^l\times H^{l-1}$ where $(s,l)\in \BR^2$ satisfies 
	\begin{equation*}
		l\ge \tfrac{d}{2}-2, \qquad \max\{ l-1, \tfrac{l}{2}+\tfrac{d-2}{4} \}\le s\le l+2, \qquad (s,l)\ne \big( \tfrac{d}{2}, \tfrac{d}{2}-2 \big), \big( \tfrac{d}{2}, \tfrac{d}{2}+1 \big).
	\end{equation*} 
	 
	\subsection{Known ill-posedness results.}
	For spatial dimension $d=1$ norm inflation was proven by Holmer in \cite[Theorem 1.2]{holmer1d} in the region $0<s<1$ and $l>2s-\tfrac{1}{2}$ or $s\le 0$ and $l>-\tfrac{1}{2}$ showing that the boundary line of the well-posed regime $l\le 2s-\tfrac{1}{2}$ is sharp.\smallskip
	
	Biagioni and Linares proved in \cite[Proposition 2.2]{Biagioni} that the Zakharov system is ill-posed in $d=1$ and for $s<0$ and $l\le -\tfrac{3}{2}$ in the sense that the initial data to solution map is not uniformly continuous. In the aforementioned work \cite[Theorem 1.3]{holmer1d} Holmer extended this result to $s=0$ and $l<-\tfrac{3}{2}$. For spatial dimension $d=3$ Ginibre, Tsutsumi and Velo showed in \cite[Proposition 3.2]{ginibre} that the solution map is not Lipschitz for $(s,l)=(\tfrac{1}{2}, -\tfrac{1}{2})$. \smallskip
	
	There are also multiple works proving $C^{2}$-ill-posedness of the Zakharov system in the sense that the second order Fréchet derivative is not continuous in the origin. Holmer proved $C^2$-ill-posedness for $s\in \BR$ and $l<-\tfrac{1}{2}$ and $d=1$ in \cite[Theorem 1.4]{holmer1d}. For $d=2$ Bejenaru, Herr, Holmer and Tataru showed $C^{2}$-ill-posedness of the Zakharov system in \cite[Theorem 1.3]{bejenaru2d} for the regime $l+\tfrac{1}{2}>2\,s$ or $l<-\tfrac{1}{2}$. Additionally, the Zakharov system exhibits $C^2$-ill-posedness in arbitrary dimension $d\in \BN$ for 
	\begin{align*}
		l<\tfrac{d}{2}-2, && s-l>2, &&2\,s -l<\tfrac{d-2}{2}, \\
		s-l<-1 && \text{or} && (s,l)= \big( \tfrac{d}{2}, \tfrac{d}{2}-2 \big), \big( \tfrac{d}{2}, \tfrac{d}{2}+1 \big).
	\end{align*}
	This was proven by Candy, Herr and Nakanishi in \cite[Section 9, Theorem 1.1]{candy2019zakharov}, which complements the aforementioned local well-posedness result with real-analytic flow map for $d\ge 4$.\medskip
	
	Using a mismatch of regularities between the Schrödinger and wave coordinate in normal form, Bejenaru, Guo, Herr and Nakanishi showed in \cite[Theorem 1.3]{zakharovindim4Herr} that there exists initial data in $H^2(\BR^4)\times H^3(\BR^4)\times H^2(\BR^4)$ such that for any $T>0$ there is no distributional solution to the Zakharov system in $L^2((0,T);H^1(\BR^4)\times H^3(\BR^4))$. 
	\subsection{New result.} Since the system \eqref{zakharovsystem} is an interaction of a nonlinear Schrödinger and wave equation, we abuse discrepancies of Sobolev regularities and the structure of the equation to show norm inflation. Thus, our result is naturally divided into Schrödinger- and wave norm inflation. Our result reads as follows.
	\begin{theorem}[Norm inflation]\label{Theorem 1.1}
	    The Zakharov system exhibits norm inflation in the wave part in dimension $d\in \BN$ for regularities $(s,l)\in \BR^2$ if
	    \begin{align*}
	    	\Big(2s<l+\frac{d-2}{2}\text{ and } s<l+\frac{1}{2}\Big) &\text{ or } \Big( 2s<l+\frac{1}{2} \text{ and }\, s<l+\tfrac{1}{2} \Big)\\
	    	& \text{ or }s+1<l\\
	    	 &\text{ or } \Big(l\le -1 \text{ and } s<-1\Big).\\
			\intertext{The Zakharov system exhibits norm inflation in the Schrödinger part for regularities $(s,l)\in \BR^2$ in dimension $d\in \BN$ if}
    		\Big( l<s-2 \text{ and } s> 0 \Big) &\text{ or } \Big( l<\tfrac{d}{2}-2,\, l+\tfrac{1}{2}\le s \text{ and } -\tfrac{d}{2}\le s< \tfrac{d}{2} \text{ but } s\ne 0 \Big)\\
    		&\text{ or } \Big(l+s<-2, \, l+\tfrac{1}{2}\le s \text{ and } s< -\tfrac{d}{2} \Big).
    	\end{align*}
	\end{theorem}
	\begin{rem}
		In sight of the aforementioned local well-posedness theory, \cref{Theorem 1.1} complements the well-posed Sobolev regularities except for some boundary lines, those $(s,l)\in \BR^2$ with $s=0$ and $l<-\tfrac{1}{2}$ as well as a bounded region for $d=1,2$ below $l<-\tfrac{1}{2}$.
	\end{rem} 
	In the following \cref{fig:plot1} we denote on the $l$-axis the wave Sobolev regularities and the $s$-axis are the Schrödinger Sobolev regularities. The light gray shaded areas excluding the dashed boundary lines are those regularity tuple for which the Zakharov system exhibits norm inflation. To classify our result, the previously mentioned local well-posedness results are illustrated in the dark gray shaded area including solid boundary lines and black dots. The black circles with white interior represent Sobolev indices for which there are no published well-posedness nor ill-posedness results in sense of Hadamard \cite{hadamardoriginal}. We want to emphasize that, in contrast to the above mentioned results on $C^2$-ill-posedness and non Lipschitz continuity of the initial data to solution map, we understand in the following plot well-posedness as the existence, uniqueness of solutions for given initial data in $H^s\times H^l\times H^{l-1}$ and continuity of the solution map. Thus, we call the system for given Sobolev regularities $(s,l)\in \BR^2$ ill-posed whenever at least one of these conditions is not met. Lastly, the black star represents the non existence of a distributional solution in $H^2(\BR^4)\times H^3(\BR^4)$ proven in \cite[Theorem 1.3]{zakharovindim4Herr} for dimension $d=4$ and for $d>4$ it shall have the same meaning as a black circle with white interior. 
	\begin{figure}[htp!]
		\centering
		\subfloat[$d=1$]{\begin{tikzpicture}
				\begin{axis}[plot1, xtick={-1,0,1,2,3,4}, ytick={-1.5,-0.5,0.5,1.5,2.5,3.5},
					xmin=0, xmax=3, ymin=-2, ymax=4]
					\addplot[domain=0:1, thick] {2*x-0.5};
					\addplot[domain=1:1.5, dashed, thick] {2*x-0.5};
					\addplot[domain=1.5:3.5, dashed, thick] {x+1};
					\addplot[domain=0:0.5, thick] {-0.5};
					\addplot[domain=0.5:1.5, dashed, thick] {-0.5};
					\addplot[domain=1.5:5, dashed, thick] {x-2};
					\fill[fill=black, opacity=0.3] (0,-0.5) -- (1.5,-0.5) -- (5,3) -- (5,4.5) -- (3.5,4.5) -- (1.5,2.5) -- cycle;
					\addplot[domain=-0.75:-0.5, dashed, thick] {-x-2};
					\addplot[domain=-0.75:0, dashed, thick] {x-0.5};
					\addplot[domain=-0.5:0.5, dashed, thick] {-1.5};
					\addplot[domain=0.5:1.5, dashed, thick] {x-2};
					\draw [dashed, very thick] (axis cs:{0,-1.5}) -- (axis cs:{0,-2});
					\fill[fill=black, opacity=0.1] (-3,5) -- (4,5) -- (1.5,2.5) -- (0,-0.5) -- (-0.75,-1.25) -- (-0.5,-1.5) -- (0.5, -1.5)-- (5,3) -- (5,-2) -- (-3,-2) -- cycle;
					\addplot[mark=*, only marks] coordinates {(0,-0.5)(0.5,-0.5)};
					\addplot[mark=*,fill=white,only marks] coordinates {(1,1.5)(1.5,-0.5)(1.5,2.5)};
				\end{axis}
		\end{tikzpicture}}
		\subfloat[$d=2$]{
			\begin{tikzpicture}
				\begin{axis}[plot1, xtick={-1,1,2,3,4}, ytick={-1.5,-0.5,0.5,1.5,2.5,3.5},
					xmin=0, xmax=3, ymin=-2, ymax=4]
					\addplot[domain=0:1.5, dashed, thick] {2*x-0.5};
					\addplot[domain=1.5:2, dashed, thick] {x+1};
					\addplot[domain=2:3.5, thick] {x+1};
					\addplot[domain=0:1.5, dashed, thick] {-0.5};
					\addplot[domain=1.5:5, dashed, thick] {x-2};
					\fill[fill=black, opacity=0.3] (0,-0.5) -- (1.5,-0.5) -- (5,3) -- (5,4.5) -- (3.5,4.5) -- (1.5,2.5) -- cycle;
					\addplot[domain=-0.5:0, dashed, thick] {x-0.5};
					\addplot[domain=-0.5:1, dashed, thick] {-1};
					\addplot[domain=1:1.5, dashed, thick] {x-2};
					\draw [dashed, very thick] (axis cs:{0,-1}) -- (axis cs:{0,-2});
					\fill[fill=black, opacity=0.1] (-3,5) -- (4,5) -- (1.5,2.5) -- (0,-0.5) -- (-0.5,-1)  -- (1, -1)-- (5,3) -- (5,-2) -- (-3,-2) -- cycle;
					\addplot[mark=*, only marks] coordinates {(0,-0.5)(2,3)};
					\addplot[mark=*,fill=white,only marks] coordinates {(1.5,-0.5)(1.5,2.5)};
				\end{axis}
		\end{tikzpicture}}\\
		\subfloat[$d=3$]{
			\begin{tikzpicture}
			\begin{axis}[plot1, xtick={-1,1,2,3,4}, ytick={-1.5,-0.5,0.5,1.5,2.5,3.5},
				xmin=0, xmax=3, ymin=-2, ymax=4]
				\addplot[domain=0:1.5, dashed, thick] {2*x-0.5};
				\addplot[domain=1.5:2, dashed, thick] {x+1};
				\addplot[domain=2:3.5, thick] {x+1};
				\addplot[domain=0:1.5, dashed, thick] {-0.5};
				\addplot[domain=1.5:5, dashed, thick] {x-2};
				\fill[fill=black, opacity=0.3] (0,-0.5) -- (1.5,-0.5) -- (5,3) -- (5,4.5) -- (3.5,4.5) -- (1.5,2.5) -- cycle;
				\draw [dashed, very thick] (axis cs:{0,-0.5}) -- (axis cs:{0,-2});
				\fill[fill=black, opacity=0.1] (-3,5) -- (4,5) -- (1.5,2.5) -- (0,-0.5)   -- (1.5, -0.5)-- (5,3) -- (5,-2) -- (-3,-2) -- cycle;
				\addplot[mark=*, only marks] coordinates {(2,3)};
				\addplot[mark=*,fill=white,only marks] coordinates {(0,-0.5)(1.5,-0.5)(1.5,2.5)};
			\end{axis}
		\end{tikzpicture}}
		\subfloat[$d\ge4$]{
			\begin{tikzpicture}
				\begin{axis}[plot1, xtick={2,3.5}, xticklabels={$\tfrac{d-3}{2}$,$\tfrac{d}{2}$}, ytick={-0.5,1.5,4.5}, yticklabels={$-0.5$,$\tfrac{d}{2}-2$,$\tfrac{d}{2}+1$},
					xmin=0, xmax=5, ymin=-1, ymax=5.5]
					\addplot[domain=2:3.5, thick] {2*x-2.5};
					\addplot[domain=3.5:6, thick] {x+1};
					\addplot[domain=2:3.5, thick] {1.5};
					\addplot[domain=3.5:7, thick] {x-2};
					\fill[fill=black, opacity=0.3] (2,1.5) -- (3.5,1.5) -- (6.5,4.5) -- (7,6.5) -- (5.5,6.5) -- (3.5,4.5) -- cycle;
					\draw [dashed, very thick] (axis cs:{0,-0.5}) -- (axis cs:{0,-2});
					\fill[fill=black, opacity=0.1] (2,1.5) -- (3.5,1.5) -- (6.5,4.5) -- (6.5,-2) -- (-2,-2) -- (-2,6.5) --  (5.5,6.5) -- (3.5,4.5) -- cycle;
					\addplot[mark=*, only marks] coordinates {(2,1.5)};
					\addplot[mark=*,fill=white, only marks] coordinates {(3.5,1.5)};
					\addplot[mark=mystar,only marks] coordinates {(3.5,4.5)};
				\end{axis}
		\end{tikzpicture}}	
		\caption{Regularities for which the Zakharov system is locally well-posed or exhibits norm inflation.}
		\label{fig:plot1}
	\end{figure}
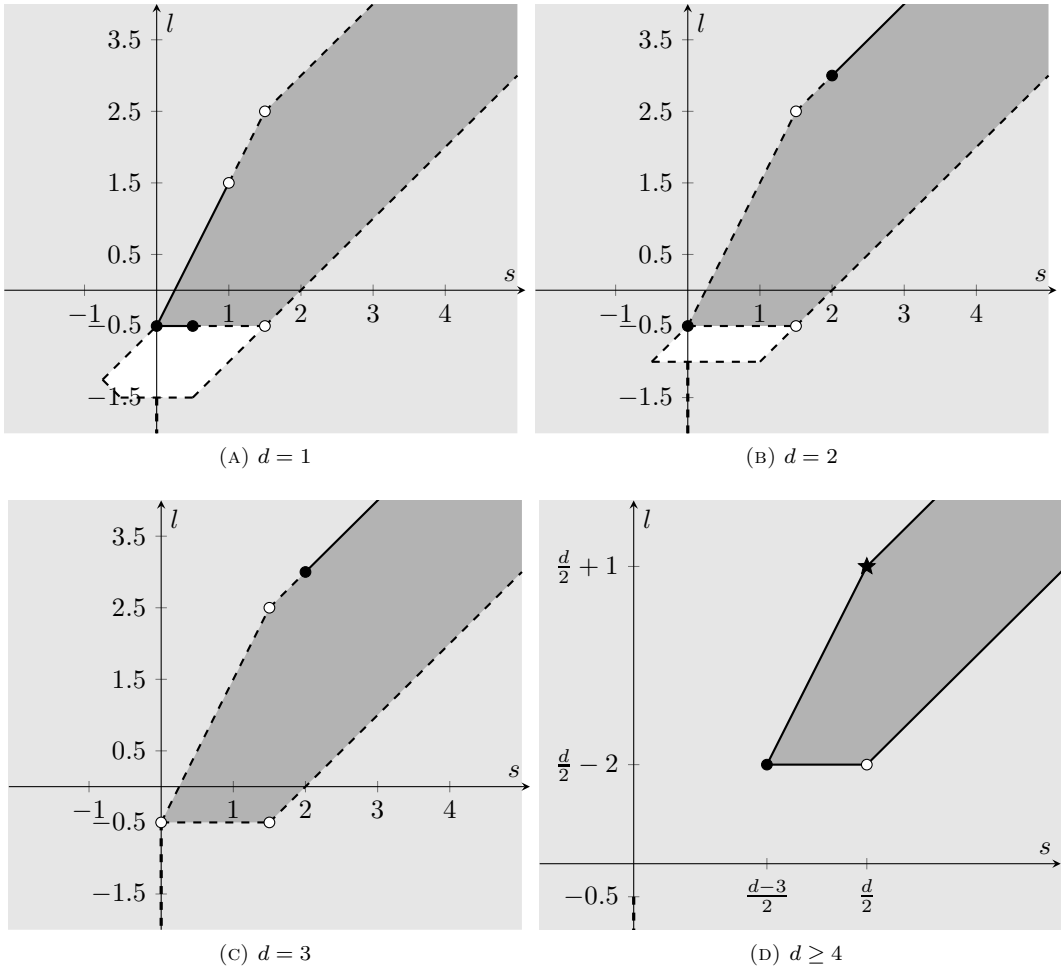\medskip

	\subsection{Outline}
		In \cref{sec:notation} we introduce notations used throughout this work. A typical reduction of the Zakharov system to a first order problem will be done in \cref{sec:duhameliteration} as well as the introduction of Duhamel's formula and a Picard iteration for the reduced equation. We begin \cref{sec:mainsection} by establishing the class of initial data we later use to show norm inflation. Thereafter, we prove upper Sobolev estimates on all terms of the Picard iteration, show a solution to our initial data exists in certain Sobolev spaces up to a time $T>0$ depending on the size of initial data in $H^{s}\times H^l$. The last ingredient for norm inflation, a lower bound on the Sobolev norm of the quadratic term in the Picard iteration, will be derived before proving \cref{Theorem 1.1} in \cref{sec:proofth}.
	\subsection*{Acknowledgments} This work is based on my master's thesis supervised by Professor Sebastian Herr. I would like to thank Sebastian Herr for fruitful discussions and helpful comments in development of this work. 
	\section{Notation and preliminaries}\label{sec:notation}Throughout this work we use the following notation. We say that $A\lesssim B$ if there exists a constant $C>0$ such that $A\le CB$ and $A\simeq B$ if $A\lesssim B$ and $B\lesssim A$. $\Re z$ is the real part and $\overline{z}$ the complex conjugate of a complex number $z\in \BC$. Additionally, $a\wedge b := \min\{ a,b \}$ and $a\vee b := \max\{ a,b \}$ for real numbers $a,b$. We write $\# A$ the number of elements in a finite set $A$. The spatial Fourier transformation of a function $f:\BR^d\to \BC$ will be written as $\SF_x f$ or $\hat{f}$ and the inverse Fourier transformation as $\SF_x^{-1}f$ or $\check{f}$. We denote by $\wo^{-s}$ the Bessel potential operator for Schwartz functions, which is defined by $\SF_x \wo^{-s} f(\xi) = (1+\abs{\xi}^2)^{-s/2} \widehat{f}(\xi)$, and as usual the Bessel potential space $H^s(\BR^d)$ as the closure of Schwartz functions with respect to the norm \begin{equation*}
		\norm{f}_{H^s} = \norm{(1+\abs{\cdot}^2)^{s/2} \widehat{f}}_{L^2}.
	\end{equation*}
	We define for short $e^{it\lap}u_0 = \SF_x^{-1} e^{-it\abs{\xi}^2} \hat{u}_0$ the Schrödinger propagator and $e^{-it \abs{\nabla}}n_0= \SF_x^{-1} e^{-it\abs{\xi}}\hat{n}_0$ the reduced wave propagator. 
	
	\subsection{Duhamel formula and Picard iteration}\label{sec:duhameliteration}
	We start by introducing a standard reduction of the wave part to a first order equation. Suppose $(u,v)$ is a solution to the Zakharov system \eqref{zakharovsystem} with initial data
	$$
	(u,v,\abs{\nabla}^{-1}\partial_t v)|_{t=0} = (u_0, v_0,v_1) \in H^s (\BR^d;\BC) \times H^l(\BR^d;\BR)\times H^l(\BR^d;\BR).
	$$
	Then $(u,n)$, where $n:= v+i\abs{\nabla}^{-1}\partial_t v$, solves the first order system
	\begin{equation}\label{eq:zakharovfirstorderalternative}
		\begin{cases}
			i \partial_t u +\lap u &= \Re(n)u\\
			i \partial_t n - \abs{\nabla}n &= \abs{\nabla}\abs{u}^2= \abs{\nabla} u \overline{u}
		\end{cases}
	\end{equation} 
	for initial data $(u, n)|_{t=0} = (u_0,v_0+iv_1)\in H^s(\BR^d;\BC)\times H^l(\BR^d; \BC), n_0= v_0+iv_1$. The reverse works similarly. We will prove norm inflation for the reduced system \eqref{eq:zakharovfirstorderalternative} with frequency localized initial data in some cube such that the wave initial part is real valued. This implies norm inflation for the original system. \medskip
	
	As in \cite{Kishimoto2019}, we will use the power series expansion of the solution to prove norm inflation. This amounts to considering the two dimensional integral equation obtained by Duhamel's trick
	\begin{equation}\label{eq:integralgleichungalternative}
		\begin{pmatrix}
			u(t)\\
			n(t)
		\end{pmatrix} = \begin{pmatrix}
			e^{it\lap}u_0\\
			e^{-it\abs{\nabla}}n_0
		\end{pmatrix}-i\int_{0}^{t}\begin{pmatrix}
			e^{i(t-s)\lap}\Re(n)(s) u(s)\\
			 e^{-i(t-s)\abs{\nabla}}\abs{\nabla} u(s)\overline{u}(s)
		\end{pmatrix} \di s=: \Phi_{\icol{u_0\\n_0} }\left(\icol{u\\n}\right)(t),
	\end{equation}
	which is the mild formulation of the reduced Zakharov system \eqref{eq:zakharovfirstorderalternative}. In alignment with the abstract ideas in \cite{Taoideaillposedness}, we will define the solution to the corresponding homogeneous problem 
\begin{equation*}
	L\left(\begin{pmatrix}
		u_0\\n_0
	\end{pmatrix}\right)(t):= \begin{pmatrix}
		e^{it\lap}&0\\
		0&e^{-it\abs{\nabla}}
	\end{pmatrix} \begin{pmatrix}
		u_0\\
		n_0
	\end{pmatrix}
\end{equation*}
and the real 2-linear Operator
	\begin{equation*}
	N\left(\begin{pmatrix}
		u_1\\n_1
	\end{pmatrix}, \begin{pmatrix}
		u_2\\
		n_2
	\end{pmatrix}\right)(t):=-i\int_{0}^{t}\begin{pmatrix}
		e^{i(t-s)\lap}&0\\
		0&e^{-i(t-s)\abs{\nabla}}
	\end{pmatrix} \begin{pmatrix}
		\Re(n_2)\, u_1\\
		\abs{\nabla} u_1\overline{u_2}
	\end{pmatrix}(s)\di s
\end{equation*}
for sufficiently regular measurable functions $u_0,u_1,u_2,n_0,n_1,n_2:\BR^d\to \BC$. Now we will construct a formal solution to the first order Zakharov system given by Picard iteration $\sum_{n=1}^\infty A_n \icol{u_0 \\ n_0}$, where the $n$-linear maps $A_n \icol{u_0\\ n_0}$ are defined recursively via 
	\begin{align}
	A_1 \icol{u_0\\ n_0}&:= L \icol{u_0\\ n_0}, \label{eq:defA1} \\
	A_n \icol{u_0\\ n_0}&:= \sum\limits_{\substack{n_1, n_2 \ge 1\\ n_1+n_2 = n}} N(A_{n_1} \icol{u_0\\ n_0}, A_{n_2} \icol{u_0\\ n_0}) \label{eq:defAn}.
	\end{align}	
	By superscripts $A_n^{(1)}$ and $A_n^{(2)}$ we refer to the projections onto the first and respectively second image coordinate.

	\section{Norm inflation}\label{sec:mainsection}
		    We take initial data localized in frequency dependent on a big $N\in \BN$ for fixed $s,l\in \BR$ 
\begin{equation}\label{eq:anfangsdaten2}
	\begin{pmatrix}
		\hat{u}_0\\
		\hat{n}_0
	\end{pmatrix} = \begin{pmatrix}
	\sum\limits_{\eta \in \Sigma_1^{(1)}} r\,\abs{Q}^{-\tfrac{1}{2}}\,\wop{\cdot}^{-s}\, \1_{\eta +Q}\\
	\sum\limits_{\eta \in \Sigma_1^{(2)}} r\,\abs{Q}^{-\tfrac{1}{2}}\,\wop{\cdot}^{-l} \, \1_{\eta + Q}
	\end{pmatrix},
\end{equation}
 where  $  \Sigma_1^{(1)} \subset \{ 0, \eta_1, \eta_2 \}$ and $\Sigma_1^{(2)} \subset \{ 0, \pm \eta_3 \}$ with $\eta_1, \eta_2, \eta_3 \in e_1\BZ \cup N e_1 \BZ\setminus\{0\}$, $e_1= (1,0,\dots, 0)\in \BR^d$ and either
\begin{align}
 Q&:= (-\tfrac{A}{2}, \tfrac{A}{2})^d \tag{typ I}\label{eq:caseI}\\
 	\intertext{or}
  Q&:= (-\tfrac{A}{2}, \tfrac{A}{2})\times (-1,1)^{d-1} \tag{typ II}\label{eq:caseII}
 \end{align}
 are $d$-dimensional open cuboids centered in zero, $A$ a positive real number with $0<A\le N$. We assume that all translations are at most of size $\abs{\eta_i}\lesssim N$ and there is at least one translation that is of order $N$. Notice that $\norm{u_0}_{H^s}, \norm{n_0}_{H^l}\lesssim r$. We choose $r:= (\log^2 N)^{-1}$ such that the initial data slowly decay to zero. Further, we define their Fourier support for short as 
 \begin{align*}
 	\Omega_1^{(1)}:= \bigcup\limits_{\eta \in \Sigma_1^{(1)}} (\eta + Q), \qquad   \Omega_1^{(2)}:= \bigcup\limits_{\eta \in \Sigma_1^{(2)}} (\eta + Q).
 \end{align*} 
Additionally, we assume $-\Omega_1^{(2)}= \Omega_1^{(2)}$ such that the spatial Fourier transformed wave initial data $n_0$ is invariant under the transform $(\cdot)\mapsto (-\,\cdot)$ and thus the wave initial data is real valued. 

We take the same approach as done in \cite[Section 3]{Taoideaillposedness} and define a formal series expansion $\sum_{n=1}^{\infty} A_n\icol{u_0\\ n_0}$ of the solution to the reduced Zakharov system \eqref{eq:zakharovfirstorderalternative} with terms $A_n\icol{u_0\\ n_0}$ as introduced in \eqref{eq:defA1} and \eqref{eq:defAn}. Under additional assumptions, we will later show that the series expansion will converge in the space of continuous functions on a compact time interval with values in product Sobolev spaces $C([0,T]; H^s\times H^l)$ and is a mild solution to the reduced Zakharov system \eqref{eq:integralgleichungalternative}. We calculate the spatial Fourier transform of the linear and quadratic terms in the expansion explicitly.
$$
\widehat{A_1}\icol{u_0\\n_0}(t,\xi)= \begin{pmatrix}
	\sum\limits_{\eta \in \Sigma_1^{(1)}} r\, \abs{Q}^{-\tfrac{1}{2}}\,\wop{\xi}^{-s} e^{-it\abs{\xi}^2}\1_{\eta +Q}\\
	\sum\limits_{\eta \in \Sigma_1^{(2)}} r\, \abs{Q}^{-\tfrac{1}{2}}\,\wop{\xi}^{-l} e^{-it\abs{\xi}}\1_{\eta+Q}
\end{pmatrix}$$
and with the consideration that the wave initial data is real valued
\begin{equation*}
	\SF_x\Re(A_1^{(2)})(t,\xi)= \sum\limits_{\eta \in \Sigma_1^{(2)}} r\, \abs{Q}^{-\tfrac{1}{2}}\,\wop{\xi}^{-l}\frac{e^{-it\abs{\xi}} +e^{it\abs{\xi}} }{2}\1_{\eta + Q}(\xi)= \cos(t\abs{\xi}) \hat{n}_0(\xi)
\end{equation*}
for all $\xi\in \BR^d$ and $t\ge 0$. We define the mixed phase functions 
\begin{equation*}
	\phi_{\pm}^{(1)}(\xi,\eta):= \abs{\xi}^2-\abs{\eta}^2\pm\abs{\xi-\eta},\qquad
	\phi^{(2)}(\xi, \eta):= \abs{\xi}-\abs{\eta}^2+\abs{\eta-\xi}^2
\end{equation*}
for any frequencies $\xi,\eta\in \BR^d$.
\begin{align*}
	\SF_x A_2^{(1)}\icol{u_0\\ n_0}(t,\xi) &= -i \int_{0}^{t}e^{-i(t-s)\abs{\xi}^2} \big(\widehat{A_1^{(1)}}(s)\ast\widehat{\Re(A_1^{(2)})}(s)\big)(\xi) \di s\nonumber\\
	&= -i  \sum\limits_{\substack{\eta_u \in \Sigma_1^{(1)} \\ \eta_n \in \Sigma_1^{(2)}}} r^2 \abs{Q}^{-1}     \int_{0}^t e^{-i(t-s)\abs{\xi}^2} \big(e^{-is\abs{\cdot}^2} \wop{\cdot}^{-s}\1_{\eta_u + Q}\ast \wop{\cdot}^{-l}\cos(s\abs{\cdot})\1_{\eta_n + Q} \big)(\xi) \di s.   \nonumber\\
	&= -i\sum\limits_{\substack{\eta_u \in \Sigma_1^{(1)} \\ \eta_n \in \Sigma_1^{(2)}}} r^2\,\abs{Q}^{-1} e^{-it\abs{\xi}^2}\int_{\eta_u +Q} \frac{\1_{\eta_n+Q }(\xi-\eta) }{\wop{\eta}^{s}\wop{\xi-\eta}^{l}}\int_{0}^{t}e^{is \phi_{+}^{(1)}(\xi,\eta)}+e^{is \phi_{-}^{(1)}(\xi,\eta)} \di s \di \eta\\
	\intertext{and}
	\SF_x A_2^{(2)}\icol{u_0\\n_0}(t,\xi)&= -i \int_{0}^{t}e^{-i(t-s)\abs{\xi}} \abs{\xi}\big( \widehat{A_1^{(1)}}(s)\ast\widehat{\overline{A_1^{(1)}}}(s) \big)(\xi) \di s\nonumber\\
	&=  -i \int_{0}^{t}e^{-i(t-s)\abs{\xi}} \abs{\xi}\big( \widehat{A_1^{(1)}}(s)\ast\overline{\widehat{A_1^{(1)}}}(s,-\,\cdot) \big)(\xi) \di s\nonumber\\
	&= -i \sum\limits_{\substack{\eta_{u,1} \in \Sigma_1^{(1)} \\ \eta_{u,2} \in \Sigma_1^{(1)}}} r^2 \abs{Q}^{-1} \abs{\xi} e^{-it\abs{\xi}} \int_{\eta_{u,1}+Q} \frac{\1_{\eta_{u,2}+Q}(\eta-\xi)}{\wop{\eta}^s \wop{\eta-\xi}^s} \int_{0}^{t} e^{is\phi^{(2)}(\xi,\eta)  } \di s \di \eta. 
\end{align*}
We will now analyze the spatial frequency support of the Fourier transformed terms $A_n\icol{u_0\\n_0}$.
\begin{lemma}\label{lem:traeger2}
	For any $n\in \BN$ and all times $t\ge 0$
	\begin{equation*}
		\supp \SF_x A_{n}^{(1)}(t) \subset \Omega_n^{(1)} := \bigcup\limits_{\eta \in \Sigma_n^{(1)}} (\eta+n Q),\quad
		\supp \SF_x A_{n}^{(2)}(t) \subset \Omega_n^{(2)} := \bigcup\limits_{\eta \in \Sigma_n^{(2)}} (\eta+n Q),
	\end{equation*}
	where \begin{align*}
	    \Sigma_n^{(1)}&:=\bigcup_{\substack{ n_1+n_2=n\\ 0\le n_1, n_2 \\ 2\nmid n_1 }}\, K_{n_1,n_2}, \quad    \Sigma_n^{(2)}:= \bigcup_{\substack{ n_1+n_2=n\\ 0\le n_1, n_2 \\ 2\mid n_1 }}\,K_{n_1,n_2}, \quad
	    K_{1,n_2}= \{ \eta \in \BR^d \,|\, \eta =  \eta_1 + \sum_{j=1}^{n_2} \tilde{\eta}_j,\, \eta_j \in \Sigma_u , \,\tilde{\eta}_j \in \Sigma_n\}
	\end{align*}
	and $K_{n_1,n_2}= \{ \eta \in \BR^d \,|\, \eta = \sum_{j=1}^{n_1}  \eta_j + \sum_{j=1}^{n_2} \tilde{\eta}_j,\, \eta_j \in \Sigma_u \cup - \Sigma_u, \,\tilde{\eta}_j \in \Sigma_n\}$ for $n_1\ne 1$.\\
	Additionally, $\#\sigma_n^{(1)}, \#\sigma_n^{(2)}\le 3^4\,n^5$ and the Lebesgue measure of the supports are bounded by 
	\begin{equation*}
		\abs{\Omega_n^{(1)}}, 	\abs{\Omega_n^{(2)}} \le   3^4 n^{d+5} \abs{Q}.
	\end{equation*}
\end{lemma}
\begin{proof}
	The proof follows easily by induction using the observations
	\begin{align*}
		\supp \SF_x A_{n}^{(1)}(t) &\subset \bigcup_{\substack{ n_1+n_2=n}} \supp\SF_x A_{n_1}^{(1)}(t) + \supp\SF_x \Re A_{n_2}^{(2)}(t)\\
		\supp \SF_x A_{n}^{(2)}(t)&\subset \bigcup_{\substack{n_1+n_2=n}} \supp \SF_x A_{n_1}^{(1)}(t) - \supp  \SF_x A_{n_2}^{(1)}(t)
	\end{align*}
	and for the induction start
	\begin{align*}
		\supp \widehat{A_1^{(1)}}(t) &\subset \bigcup\limits_{\eta \in \Sigma_1^{(1)}} (\eta + Q) = \Omega_1^{(1)},\qquad\qquad\,\,\,\,
		\supp \widehat{A_1^{(2)}}(t) \subset \bigcup\limits_{\eta \in \Sigma_1^{(2)}} (\eta + Q)= \Omega_1^{(2)},\\
		\supp \widehat{A_2^{(1)}}(t)&\subset \bigcup\limits_{\substack{ \eta_1 \in \Sigma_1^{(1)} \\ \eta_2 \in \Sigma_1^{(2)} }} (\eta_1+\eta_2 +2Q),\qquad\qquad
        \supp \widehat{A_2^{(2)}}(t)\subset \bigcup\limits_{\eta_1,\eta_2 \in \Sigma_1^{(1)}} (\eta_1-\eta_2 + 2Q).
    \end{align*}
	With the restrictions $\# \Sigma_1^{(1)}\le 3$, $\# \Sigma_1^{(2)} \le 3$ and $-\Sigma_1^{(2)} = \Sigma_1^{(2)}$, the number of possible combinations in $K_{n_1,n_2}$ for $n\in \BN$ and nonnegative $n_1, n_2$ with $n_1+n_2=n$ is at most
\begin{align*}
    \# K_{n_1,n_2}\le (2n_1+1)^3(2n_2+1)\le (2n+1)^4\le (3n)^4.
\end{align*}
Thus $\#\sigma_n^{(1)}, \#\sigma_n^{(2)}\le 3^4\,n^5$ and the claimed Lebesgue measure bounds of those support estimates follow by an easy calculation.
\end{proof}
The fact that the following sequence grows at most exponentially will be helpful while proving $L^2$-bounds for both coordinates of each term of $\sum_{n=1}^{\infty}A_n\icol{u_0\\n_0}$.
\begin{lemma}\label{prop:summe2}
	Define iteratively $b_1^{(1)}= b_1^{(2)}:= 1$ and
	\begin{align*}
		b_n^{(1)}:= \frac{1}{n-1}\sum\limits_{n_1+n_2=n} (n_1 \wedge n_2)^{\tfrac{d+5}{2}} b_{n_1}^{(1)} b_{n_2}^{(2)},\qquad
		b_n^{(2)}:= \frac{n}{n-1}\sum\limits_{n_1+n_2=n} (n_1\wedge n_2)^{\tfrac{d+5}{2}} b_{n_1}^{(1)} b_{n_2}^{(1)}.
	\end{align*}
	There exists a constant $C_b>0$ such that $b_n^{(1)}, b_n^{(2)}\le C_b^n$.
\end{lemma}
\begin{proof}
    By \cref{prop:sequence} in the appendix with $C_1=C_2=1$, $\alpha:= \frac{d+5}{2}$ and starting values $b_1^{(1)}=b_1^{(2)}=1$, 
    \begin{align*}
        b_n^{(1)}, b_n^{(2)}\le n^{-\tfrac{d+5}{2}-2} \left( 2^{\tfrac{d+5}{2}+5} \right)^{n-1} \le \left( 2^{\tfrac{d+1}{2}+7} \right)^{n-1}.
    \end{align*}
\end{proof}
Now we establish $L^2$ bounds on each term $A_n$. We handle the cases $nQ$ and $\Omega_{n}^{(1)}\setminus nQ$, $\Omega_{n}^{(2)}\setminus nQ$ and the Schrödinger and wave coordinate separately.
\begin{lemma}[$L^2$-bounds]\label{lem:L2bounds}
	Fix $d\in\BN$, $s,l\in \BR$, $n\in \BN$, $T>0$. At any time $0\le t\le T$ the following bounds on the coordinates of the $n$-th addend $A_n$ hold. There exists a constant $C>0$, which depends on $d$ but not on $s, l, t,T,Q, N$ and $n$, such that for any $0<t\le T$
	\begin{align*}
		\norm{\SF_x A_n^{(1)}(t)}_{L^2(nQ)}&\le C^n\abs{Q}^{\tfrac{n-1}{2}}T^{n-1}  \max\limits_{\substack{1\le k\le n\\ 2\nmid k }}\Big\{ N^{\tfrac{k-1}{2}}(\norm{\hat{u}_0}_{L^2})^k (\norm{\hat{n}_0}_{L^2})^{n-k}  \Big\},\\
		\norm{\SF_x A_n^{(1)}(t)}_{L^2((nQ)^c)}&\le C^n\abs{Q}^{\tfrac{n-1}{2}} T^{n-1}  \frac{ \big( N^{\tfrac{1}{2}}\norm{\hat{u}_0}_{L^2(Q^c)}\big)\vee \norm{\hat{n}_0}_{L^2(Q^c)}  }{\big(N^{\tfrac{1}{2}} \norm{\hat{u}_0}_{L^2}\big) \vee \norm{\hat{n}_0}_{L^2}} \max\limits_{\substack{1\le k\le n\\ 2\nmid k }} \Big\{ N^{\tfrac{k-1}{2}}(\norm{\hat{u}_0}_{L^2})^k (\norm{\hat{n}_0}_{L^2})^{n-k}  \Big\},\\
		\norm{\SF_x A_n^{(2)}(t)}_{L^2(nQ)}&\le C^n \abs{Q}^{\tfrac{n-1}{2}}T^{n-1}   \max\limits_{\substack{2\le k\le n\\ 2\mid k}}\Big\{ N^{\tfrac{k}{2}}(\norm{\hat{u}_0}_{L^2})^k (\norm{\hat{n}_0}_{L^2})^{n-k} \Big\},\\
		\norm{\SF_x A_n^{(2)}(t)}_{L^2((nQ)^c)}&\le C^n\abs{Q}^{\tfrac{n-1}{2}}T^{n-1} \frac{ \big(  N^{\tfrac{1}{2}}\norm{\hat{u}_0}_{L^2(Q^c)}\big) \vee \norm{\hat{n}_0}_{L^2(Q^c)} }{\big(N^{\tfrac{1}{2}} \norm{\hat{u}_0}_{L^2}\big)\vee \norm{\hat{n}_0}_{L^2} } \max\limits_{\substack{2\le k\le n\\ 2\mid k}}   \Big\{ N^{\tfrac{k}{2}}(\norm{\hat{u}_0}_{L^2})^k (\norm{\hat{n}_0}_{L^2})^{n-k}\Big\}
	\end{align*}
	with obvious modifications for $n=1$. 
\end{lemma}
\begin{proof}
	Instead of proving the above inequalities directly, we show the statement with $C^n$ replaced by $c^{n-1}\,b_n^{(1)}$ for the first coordinate and $c^{n-1}\,b_n^{(2)}$ for the second coordinate from \cref{prop:summe2} and $c=3^3(d+1)$. We proceed inductively. The case $n=1$ is trivial. To allow for a better overview we define for short $a_{1}^{(1)}:= \norm{\hat{u}_0}_{L^2}$, $a_{1,c}^{(1)}:= \norm{\hat{u}_0}_{L^2(Q^c)}$, $a_{1}^{(2)}:= \norm{\hat{n}_0}_{L^2}$ and $a_{1,c}^{(2)}:= \norm{\hat{n}_0}_{L^2(Q^c)}$ and for $n\ge 2$
	\begin{align*}
		a_{n}^{(1)}&:= \max\limits_{\substack{1\le k\le n\\ 2\nmid k }}\Big\{ N^{\tfrac{k-1}{2}}(\norm{\hat{u}_0}_{L^2})^k (\norm{\hat{n}_0}_{L^2})^{n-k}  \Big\}, \quad a_{n}^{(2)}:= \max\limits_{\substack{2\le k\le n\\ 2\mid k}}\Big\{ N^{\tfrac{k}{2}}(\norm{\hat{u}_0}_{L^2})^k (\norm{\hat{n}_0}_{L^2})^{n-k} \Big\},\\
		a_{n,c}^{(1)}&:= \frac{ \max\{ N^{\tfrac{1}{2}}\norm{\hat{u}_0}_{L^2(Q^c)},\norm{\hat{n}_0}_{L^2(Q^c)} \}}{\max\{N^{\tfrac{1}{2}} \norm{\hat{u}_0}_{L^2},\norm{\hat{n}_0}_{L^2} \}} \max\limits_{\substack{1\le k\le n\\ 2\nmid k }}\Big\{ N^{\tfrac{k-1}{2}}(\norm{\hat{u}_0}_{L^2})^k (\norm{\hat{n}_0}_{L^2})^{n-k}  \Big\},\\
		a_{n,c}^{(2)}&:= \frac{ \max\{ N^{\tfrac{1}{2}} \norm{\hat{u}_0}_{L^2(Q^c)},\norm{\hat{n}_0}_{L^2(Q^c)} \}}{\max\{N^{\tfrac{1}{2}} \norm{\hat{u}_0}_{L^2},\norm{\hat{n}_0}_{L^2} \}} \max\limits_{\substack{2\le k\le n\\ 2\mid k}}\Big\{ N^{\tfrac{k}{2}}(\norm{\hat{u}_0}_{L^2})^k (\norm{\hat{n}_0}_{L^2})^{n-k}\Big\}.
	\end{align*}
	We begin by making two observations. Firstly, $a_{n,c}^{(1)}\le a_n^{(1)} $ and $a_{n,c}^{(2)} \le a_{n}^{(2)}$ for any $n\in \BN$ and we notice by simple case distinction that for natural numbers $n_1, n_2, n$ with $n_1+n_2 =n$
	\begin{align*}
		a_{n_1}^{(1)}a_{n_2}^{(2)} &\le a_n^{(1)}, &&&
		a_{n_1}^{(1)}a_{n_2,c}^{(2)} &\le a_{n,c}^{(1)}, &&& 	a_{n_1,c}^{(1)}a_{n_2}^{(2)}&\le a_{n,c}^{(1)}, \\
		N a_{n_1}^{(1)}a_{n_2}^{(1)}&\le  a_{n}^{(2)}, &&&  N a_{n_1}^{(1)}a_{n_2,c}^{(1)}&\le a_{n,c}^{(2)} .
	\end{align*}  
	 Now suppose the claim is true for $n-1\in \BN$. Thereby,
	\begin{align*}
		\norm{\SF_x A_n^{(1)}(t)}_{L^2(nQ)}&\le  \sum\limits_{\substack{n_1,n_2\ge 1\\ n_1+n_2=n}} \int_{0}^{T}\norm{\widehat{A_{n_1}^{(1)}}(\tilde{t})\ast\widehat{\Re A_{n_2}^{(2)}}(\tilde{t})}_{L^2(nQ)} \di \tilde{t}.\\
		\intertext{Now we apply Young inequality with the smaller index $n_1, n_2$ in $L^1$ and Hölder inequality.}
		&\le  \sum\limits_{\substack{n_1,n_2\ge 1\\ n_1+n_2=n}} \abs{\Omega_{\min\{ n_1,n_2 \}}}^{1/2}\int_{0}^{T}  \norm{\widehat{A_{n_1}^{(1)}}(\tilde{t})}_{L^2}\norm{\widehat{A_{n_2}^{(2)}}(\tilde{t})}_{L^2} \di \tilde{t}.\\
		\intertext{Further, \cref{lem:traeger2} and the induction assumption yield}
		\norm{\SF_x A_n^{(1)}(t)}_{L^2(nQ)}&\le \sum\limits_{\substack{n_1,n_2\ge 1\\ n_1+n_2=n}} 3^2 (\min\{ n_1,n_2\})^{\tfrac{d+5}{2}} a_{n_1}^{(1)} b_{n_1}^{(1)} a_{n_2}^{(2)} b_{n_2}^{(2)} c^{n-2} \abs{Q}^{\tfrac{n-1}{2}}\int_{0}^{T} {\tilde{t}}^{n-2} \di \tilde{t} \\
		&\le  c^{n-1} b_{n}^{(1)} \abs{Q}^{\tfrac{n-1}{2}}T^{n-1} a_{n}^{(1)} .
	\end{align*}
	For $\norm{\SF_x A_n^{(1)} (t) }_{L^2(nQ^c)}$ we proceed similarly. The only difference is an additional distinction in the application of Young inequality into $nQ$ and $nQ^c$. We demonstrate this in the wave coordinate.
	\begin{equation*}
		\norm{\SF_x A_n^{(2)}(t)}_{L^2(nQ^c)}\le  \sum\limits_{\substack{n_1,n_2\ge 1\\ n_1+n_2=n}} \int_{0}^{T} \norm{\abs{\cdot}\,\,\widehat{A_{n_1}^{(1)}}(\tilde{t})\ast\widehat{\overline{A_{n_2}^{(1)}}}(\tilde{t})}_{L^2(nQ^c)} \di \tilde{t}.
	\end{equation*}
	We continue by estimating $\abs{\cdot}$ by its maximum in $\Omega_{n}^{(2)}$. Thereafter, we apply Young with the smaller index in $L^1$ as done above. Notice for the convolution to be in $\Omega_{n}^{(2)} \setminus nQ$ it is necessary that at least one term contributes high frequencies from $nQ^c$. Thus, after Hölder inequality we receive the following upper bound.
	\begin{align*}
		\norm{\SF_x A_n^{(2)}(t)}_{L^2(nQ^c)}&\le  \sum\limits_{\substack{n_1,n_2\ge 1\\ n_1+n_2=n}} \abs{\Omega_{\min\{ n_1,n_2 \}}}^{1/2} (d+1)nN \int_{0}^{T} \Bigg( \norm{\widehat{A_{n_1}^{(1)}}(\tilde{t})}_{L^2(nQ^c)}\norm{\widehat{A_{n_2}^{(1)}}(\tilde{t}, -\cdot)}_{L^2(nQ^c)} \\
		&\qquad\quad + \norm{\widehat{A_{n_1}^{(1)}}(\tilde{t})}_{L^2(nQ)}\norm{\widehat{A_{n_2}^{(1)}}(\tilde{t}, -\cdot)}_{L^2(nQ^c)}+\norm{\widehat{A_{n_1}^{(1)}}(\tilde{t})}_{L^2(nQ^c)}\norm{\widehat{A_{n_2}^{(1)}}(\tilde{t}, -\cdot)}_{L^2(nQ)}\Bigg)\di \tilde{t}.\\
		\intertext{\cref{lem:traeger2} and the induction assumption yield}
		\norm{\SF_x A_n^{(2)}(t)}_{L^2(nQ^c)}&\le \sum\limits_{\substack{n_1,n_2\ge 1\\ n_1+n_2=n}} 3^2 (\min\{ n_1,n_2\})^{\tfrac{d+5}{2}} (d+1)nN \big( a_{n_1,c}^{(1)}  a_{n_2,c}^{(1)}  +a_{n_1}^{(1)}  a_{n_2,c}^{(1)} +a_{n_1,c}^{(1)}  a_{n_2}^{(1)}  \big)\\
		&\qquad \qquad \cdot b_{n_1}^{(1)}b_{n_2}^{(1)} c^{n-2} \abs{Q}^{\tfrac{n-1}{2}}\int_{0}^{T} \tilde{t}^{n-2} \di \tilde{t} \le c^{n-1} b_{n}^{(2)} \abs{Q}^{\tfrac{n-1}{2}} T^{n-1} a_{n,c}^{(2)}.
	\end{align*}
	Again the estimate on $\norm{\SF_x A_n^{(2)}(t)}_{L^2(nQ)}$ is quite similar except that the case distinction before applying Young's inequality is not needed. Thus by \cref{prop:summe2}, the statement is proven.
\end{proof}
Now we use these $L^2$ bounds to derive Sobolev bounds of each term of the series expansion. 
\begin{lemma}[Sobolev bounds]\label{lem:Sobolevbounds}
	Fix any $d\in\BN$, $s,l\in \BR$, $n\ge 2$, $T>0$. At any time $0\le t\le T$ the following bounds on the coordinates of the $n$-th addend $A_n$ hold. There exists a constant $C>0$, which depends only on $d$ but not on $s, l, t,T,Q, N$ and $n$, such that
	\begin{align*}
		\norm{A_n^{(1)}(t)}_{H^s}&\le C^n\abs{Q}^{\tfrac{n}{2}-1}T^{n-1} \, \max\limits_{\substack{1\le k\le n\\ 2\nmid k }}\Big\{ N^{\tfrac{k-1}{2}}(\norm{\hat{u}_0}_{L^2})^k (\norm{\hat{n}_0}_{L^2})^{n-k} \Big\}\\
		&\qquad\qquad \cdot \Bigg(  \norm{\wop{\cdot}^s}_{L^2(nQ)} + \norm{\wop{\cdot}^s}_{L^2(\Omega_{n}^{(1)}\setminus nQ)}\,\frac{ \max\{ N^{\tfrac{1}{2}}\norm{\hat{u}_0}_{L^2(nQ^c)},\norm{\hat{n}_0}_{L^2(nQ^c)} \}}{\max\{N^{\tfrac{1}{2}} \norm{\hat{u}_0}_{L^2},\norm{\hat{n}_0}_{L^2} \}}  \Bigg),\\
		\norm{A_n^{(2)}(t)}_{H^l}&\le C^n \abs{Q}^{\tfrac{n}{2}-1}T^{n-1} \,  \max\limits_{\substack{2\le k\le n\\ 2\mid k}}\Big\{ N^{\tfrac{k}{2}-1}(\norm{\hat{u}_0}_{L^2})^k (\norm{\hat{n}_0}_{L^2})^{n-k} \Big\}\\
		&\qquad\qquad \cdot \Bigg(  \norm{\wop{\cdot}^l\abs{\cdot}}_{L^2(nQ)} + \norm{\wop{\cdot}^l\abs{\cdot}}_{L^2(\Omega_{n}^{(2)}\setminus nQ)}\frac{ \max\{ N^{\tfrac{1}{2}}\norm{\hat{u}_0}_{L^2(nQ^c)},\norm{\hat{n}_0}_{L^2(nQ^c)} \}}{\max\{N^{\tfrac{1}{2}} \norm{\hat{u}_0}_{L^2},\norm{\hat{n}_0}_{L^2} \}}  \Bigg).
	\end{align*}
\end{lemma}
\begin{proof}
	Since $e^{it\,\lap}$ and $e^{-it\abs{\nabla}}$ are unitary in $L^2$, we calculate employing Hölder and Young inequality
	\begin{align*}
		\norm{A_n^{(1)}(t)}_{H^s}&\le \sqrt{2}\norm{\wop{\cdot}^s}_{L^2(nQ)}\norm{\SF_x A_n^{(1)}(t)}_{L^\infty(nQ)}+\sqrt{2} \norm{\wop{\cdot}^s}_{L^2(nQ^c)}\norm{\SF_x A_n^{(1)}(t)}_{L^\infty(nQ^c)}\\
		&\le \sqrt{2}\norm{\wop{\cdot}^s}_{L^2(nQ)}T\sum\limits_{\substack{ n_1+n_2=n\\ 1\le n_1, n_2 \le n}} \norm{\SF_x A_{n_1}^{(1)}(t) \ast\SF_x A_{n_2}^{(2)}(t)   }_{L^\infty(nQ)}\\
		&\quad+ \sqrt{2}\norm{\wop{\cdot}^s}_{L^2(nQ^c)}T\sum\limits_{\substack{ n_1+n_2=n\\ 1\le n_1, n_2 \le n}} \norm{\SF_x A_{n_1}^{(1)}(t) \ast\SF_x A_{n_2}^{(2)}(t)   }_{L^\infty(nQ^c)}\\
		&\le \sqrt{2}\norm{\wop{\cdot}^s}_{L^2(nQ)}T\sum\limits_{\substack{ n_1+n_2=n\\ 1\le n_1, n_2 \le n}} \norm{\SF_x A_{n_1}^{(1)}(t)}_{L^2} \norm{\SF_x A_{n_2}^{(2)}  (t)}_{L^2}\\
		&\quad+ \sqrt{2}\norm{\wop{\cdot}^s}_{L^2(nQ^c)}T\sum\limits_{\substack{ n_1+n_2=n\\ 1\le n_1, n_2 \le n}}\Big(\norm{\SF_x A_{n_1}^{(1)}(t)}_{L^2(n_1 Q^c)} \norm{\SF_x A_{n_2}^{(2)}(t)  }_{L^2(n_2 Q)}\\
		&\qquad\quad+\norm{\SF_x A_{n_1}^{(1)}(t)}_{L^2(n_1 Q)} \norm{\SF_x A_{n_2}^{(2)}(t)  }_{L^2(n_2 Q^c)}+\norm{\SF_x A_{n_1}^{(1)}(t)}_{L^2(n_1 Q^c)} \norm{\SF_x A_{n_2}^{(2)} (t) }_{L^2(n_2 Q^c)}\Big).
	\end{align*}
	By \cref{lem:L2bounds}, we receive the desired bound with an additional factor $\sqrt{2}\,3 n$ which we estimate by $(3e)^n$ and absorb it into $C^n$. The proof of the second inequality works exactly the same and we omit it. 
\end{proof}
\begin{lem}[Bessel potential $L^2$-estimates]\label{lem:besselpot_estimates}
		Let $\Omega_{n}^{(1)}, \Omega_{n}^{(2)}$ be the supersets of the supports of the Schrödinger coordinate and resp. wave coordinate of the $n$-th term as given by \cref{lem:traeger2}. For any $(s,l)\in \BR^2$
		\begin{align*}
			\norm{\wop{\cdot}^{s}}_{L^2(\Omega_{n}^{(1)}) } &\le3^2 \big(2^{\tfrac{d+5}{2}+\abs{s}}\big)^n \norm{\wop{\cdot}^{s}}_{L^2(Q\cup (Ne_1+Q))},\\
			\norm{\wop{\cdot}^{l} \, \abs{\cdot}}_{L^2(\Omega_{n}^{(2)}) } &\le 3^2 \big(2^{\tfrac{d+5}{2}+\abs{l}+1}\big)^n \norm{\wop{\cdot}^{l}\,\abs{\cdot}}_{L^2(Q\cup (Ne_1+Q))}.
		\end{align*}
\end{lem}
\begin{proof}
	Firstly, notice that for any $n\in \BN$ and any $(s,l)\in \BR^2$
	\begin{align*}
		\norm{\wop{\cdot}^{s}}_{L^2(\Omega_{n}^{(1)}) }^2 \le \sum\limits_{\eta \in \Sigma_n^{(1)}} \int\limits_{\eta+nQ} (1+\abs{\xi}^2)^{s} \di \xi =\sum\limits_{\eta \in \Sigma_n^{(1)}} n^d \int\limits_{\eta/n+Q} (1+n^2\abs{\xi}^2)^{s} \di \xi \le \sum\limits_{\eta \in \Sigma_n^{(1)}} n^{d+2\,\abs{s}} \int\limits_{\eta/n+Q} (1+\abs{\xi}^2)^{s} \di \xi.
	\end{align*}
	The translations $\eta\in \Sigma_n^{(1)}$ are at most $\abs{\eta}\le nN$. By the monotonicity of $r\mapsto (1+r^2)^{s}$ and $\# \Sigma_n^{(1)}\le 3^4n^5$, see \cref{lem:traeger2}, we estimate further
	\begin{align*}
		\norm{\wop{\cdot}^{s}}_{L^2(\Omega_{n}^{(1)}) } \le 3^2 n^{\tfrac{d+5}{2}+\abs{s}} \norm{\wop{\cdot}^{s}}_{L^2(Q\cup (Ne_1+Q))} \le 3^2 \big(2^{\tfrac{d+5}{2}+\abs{s}}\big)^n \norm{\wop{\cdot}^{s}}_{L^2(Q\cup (Ne_1+Q))}.
	\end{align*}
	In a similar fashion, we receive \begin{equation*}
		\norm{\wop{\cdot}^{l} \, \abs{\cdot}}_{L^2(\Omega_{n}^{(2)}) } \le 3^2 \big(2^{\tfrac{d+5}{2}+\abs{l}+1}\big)^n \norm{\wop{\cdot}^{l}\,\abs{\cdot}}_{L^2(Q\cup (Ne_1+Q))}.
	\end{equation*}
\end{proof}
\begin{lem}[Existence of the solution]\label{lem:existenz}
	Let $s,l\in \BR$, $d\in \BN$ and suppose the initial data $u_0$ and $n_0$ are as in \eqref{eq:anfangsdaten2} w.r.t. $s$ and $l$. Then for any time $T>0$ such that 
	\begin{equation*}
		\rho:=C\,2^{\tfrac{d+7}{2}+\abs{l} +s_\star}\abs{Q}^{\tfrac{1}{2}} \max\{ N^{\tfrac{1}{2}}\, \norm{\hat{u}_0}_{L^2}, \norm{\hat{n}_0}_{L^2} \} \, T <1,
	\end{equation*}
	where $C$ is the constant from \cref{lem:Sobolevbounds} and 
	\begin{equation*}
		s_\star := \max\Big\{  \frac{d+1}{2}, \frac{\abs{l}}{2}+ \frac{d}{4}+1, \abs{s} \Big\},
	\end{equation*} the series $\sum_{n=1}^{\infty}A_n\icol{u_0\\n_0}$ converges absolutely in $C([0,T]; H^{s,l})$ and is a solution to the reduced Zakharov system in mild formulation \eqref{eq:integralgleichungalternative}. 
\end{lem}
\begin{proof}
	Fix $s,l\in \BR$, $u_0$ and $n_0$ as in \eqref{eq:anfangsdaten2} w.r.t. $s$ and $l$ and a time $T>0$ such that $\rho<1$. Firstly, notice that for any $n\in \BN$ and any $(s',l')\in \{(s,l),(\abs{s}, \abs{l}),(s_\star, \abs{s})\}$ by \cref{lem:besselpot_estimates}
	\begin{align*}
		\norm{\wop{\cdot}^{s'}}_{\Omega_{n}^{(1)}}^2 &\le 3^2 \big(2^{\tfrac{d+5}{2}+\abs{s'}}\big)^n \norm{\wop{\cdot}^{s'}}_{L^2(Q\cup (Ne_1+Q))},\\
		\norm{\wop{\cdot}^{l'} \, \abs{\cdot}}_{\Omega_{n}^{(2)}} &\le 3^2 \big(2^{\tfrac{d+5}{2}+\abs{l'}+1}\big)^n \norm{\wop{\cdot}^{l'}\,\abs{\cdot}}_{L^2(Q\cup (Ne_1+Q))}.
	\end{align*}
	By the above calculations and \cref{lem:Sobolevbounds}, there exists a constant $C_{s,l,d,Q, N, u_0, n_0}>0$, which may depend on $s,l,d,Q, N$ as well as $u_0$, $n_0$, $\norm{\wop{\cdot}^{s'}}_{L^2(Q\cup (Ne_1+Q))}$ and $\norm{\wop{\cdot}^{l'}\abs{\cdot}}_{L^2(Q\cup (Ne_1+Q))}$ for all three choices of $(s',l')$ but not on $n$, such that 
	\begin{equation}\label{eq:general_Sobolev_bound}
		\norm{A_n \icol{u_0\\n_0}}_{C ([0,T];H^{s'}\times H^{l'})}\le C_{s,l,d,Q, N, u_0, n_0}\, \rho^{n-1}
	\end{equation}
	holds for all three choices of $(s',l')$. Since we chose $T>0$ such that $\rho<1$, the series $\sum_{n=1}^\infty A_n\icol{u_0\\n_0}$ converges absolutely in $C ([0,T];H^{s'}\times H^{l'})$ for all choices of $(s',l')$. Now pick \begin{equation*}
		\delta':= 2\,C_{s,l,d,Q, N, u_0, n_0}\, \frac{1}{1-\rho},
	\end{equation*} 
	which is chosen such that $\norm{\sum_{n=1}^{\infty} A\icol{u_0\\n_0}}_{H^{s_\star,\abs{s}}}<\delta'$. With this choice and by noticing that the initial data are sufficiently regular $u_0,n_0\in H^\infty$, the operator $\Phi_{\icol{u_0\\ n_0}}$ is a continuous map between the ball of radius $\delta'$ centered in zero in $C([0,T]; H^{s_\star, \abs{s}})$ and the space $C([0,T]; H^{\abs{s}, \abs{l}})$ by \cref{lem:contsolopabove}. We define for short $u_m:= \sum_{n=1}^{m} A_n\icol{u_0\\n_0}$ and calculate
	\begin{align*}
		u_m-\Phi_{\icol{u_0\\n_0}} u_m &= u_m-A_1\icol{u_0\\n_0}- N(u_m,u_m) = u_m- A_1\icol{u_0\\n_0}- \sum_{n_1,n_2=1}^{m} N(A_{n_1}\icol{u_0\\n_0},A_{n_2}\icol{u_0\\n_0})\\
		&= u_m- A_1\icol{u_0\\n_0}- \sum_{n=2}^m A_n \icol{u_0\\n_0}-\sum_{\substack{n_1,n_2=1\\ n_1+n_2>m}}^{m} N(A_{n_1}\icol{u_0\\n_0},A_{n_2}\icol{u_0\\n_0}) \\
		&= -\sum_{\substack{n_1,n_2=1\\ n_1+n_2>m}}^{m} N(A_{n_1}\icol{u_0\\n_0},A_{n_2}\icol{u_0\\n_0}).
	\end{align*}
	Applying similar calculations as done to retrieve \eqref{eq:general_Sobolev_bound}, this converges to zero in $C([0,T]; H^{\abs{s}, \abs{l}})$ by 
	\begin{equation*}
		\norm{u_m - \Phi_{\icol{u_0\\ n_0}}u_m}_{C ([0,T];H^{\abs{s}, \abs{l}})}\lesssim C_{s,l,d,Q, N, u_0, n_0}\, \sum_{n=m+1}^{2m}\rho^{n-1} \to 0 \text{ as }m \to \infty.
	\end{equation*}
	Now we use the continuity of $\Phi_{\icol{u_0\\ n_0}}$ to conclude that $\sum_{n=1}^{\infty}A_n\icol{u_0\\n_0}$ solves the reduced Zakharov system in mild formulation.
\end{proof}
A crucial step in proving norm inflation is a proper lower Sobolev bound on the solution. To achieve this, we prove lower bounds on the quadratic terms $A_2\icol{u_0\\n_0}$ and later show that they dominate the solution. In the following lemma we will also specify different sub types of initial data.
\begin{lemma}\label{lem:lower_sobo_bounds}
	Let $d\in \BN$ and $s,l\in \BR$. Consider initial data $u_0$ and $n_0$ as in \eqref{eq:anfangsdaten2}. For sufficiently big $N\in \BN$ the following lower bound on the second term in the series expansion are true.
	\begin{itemize}
		\item[\namedlabel{itm:a}{\normalfont\textbf{(a)}}]{\textit For initial data of \eqref{eq:caseI} with $A= \tfrac{N}{\log(N)}$, $\Sigma_1^{(1)}=\{\pm Ne_1\}$, $\Sigma_1^{(2)}= \emptyset$ and times $0<T\ll N^{-2} \log(N)$ 
			\begin{equation*}
				\norm{A_2^{(2)}(T)}_{H^l}\gtrsim r^2 \, T \, N^{l-2\,s+1}\, A^{\tfrac{d}{2}}. 
			\end{equation*}
		}
		\item[\namedlabel{itm:b}{\normalfont \textbf{(b)}}]{\textit For initial data of \eqref{eq:caseII} with $A=\tfrac{1}{N}$, $\Sigma_1^{(1)}=\{(N+1)e_1, -Ne_1\}$, $\Sigma_1^{(2)}= \emptyset$ and times $0<T\ll 1$
			\begin{equation*}
				\norm{A_2^{(2)}(T)}_{H^l}\gtrsim r^2\,T\, N^{l-2\,s+\tfrac{1}{2}}.
			\end{equation*}
		}
		\item[\namedlabel{itm:c}{\normalfont\textbf{(c)}}]{\textit For initial data with $A=1$, $\Sigma_1^{(1)}=\{0, Ne_1\}$, $\Sigma_1^{(2)}= \emptyset$ and times $0<T\ll N^{-2}$
			\begin{align*}
				\norm{A_2^{(2)}(T)}_{H^l}\gtrsim r^2\,T\,N^{l-s+1}.
			\end{align*}
			\begin{itemize}
				\item[\namedlabel{itm:c1}{\normalfont{\textbf{(c')}}} ]{\textit
					And for times $0<T\ll N^{-1}$
					\begin{align*}
						\norm{A_2^{(2)}(T)}_{H^l}\gtrsim r^2\,T\,N^{-2\,s}.
					\end{align*}
				}
			\end{itemize}
		}
		\item[\namedlabel{itm:d}{\normalfont\textbf{(d)}}]{\textit For initial data with $A=1$, $\Sigma_1^{(1)}=\{0\}$, $\Sigma_1^{(2)}= \{ \pm Ne_1 \}$ and times $0<T\ll N^{-2}$
			\begin{equation*}
				\norm{A_2^{(1)}(T)}_{H^s}\gtrsim r^2\,T\,N^{s-l}.
			\end{equation*}
		}
		\item[\namedlabel{itm:e}{\normalfont\textbf{(e)}}]{\textit For initial data of \eqref{eq:caseI} with $A=\tfrac{N}{\log(N)}$, $\Sigma_1^{(1)}=\{\pm N e_1\}$, $\Sigma_1^{(2)}= \{ \pm Ne_1 \}$ and times $0<T\ll N^{-2}$
			\begin{equation*}
				\norm{A_2^{(1)}(T)}_{H^s}\gtrsim r^2\,T\,N^{-l-s}\, \begin{cases}
					A^{\tfrac{d}{2}+s} &,\tfrac{d}{2}+s>0\\
					\ln(A)&,\tfrac{d}{2}+s=0\\
					1&,\tfrac{d}{2}+s<0
				\end{cases}.
			\end{equation*}
		}
		\item[\namedlabel{itm:f}{\normalfont\textbf{(f)}}]{\textit For initial data of \eqref{eq:caseI} with $A=\tfrac{N}{\log(N)}$, $\Sigma_1^{(1)}=\{0\}$, $\Sigma_1^{(2)}= \{ \pm Ne_1 \}$ and times $0<T\ll N^{-2}$
		\begin{equation*}
			\norm{A_2^{(1)}(T)}_{H^s}\gtrsim r^2 T N^{s-l}  \begin{cases}
				A^{\tfrac{d}{2}-s} &, d-s >0\\
				A^{-\tfrac{d}{2}} \log (A) &, d-s=0\\
				A^{-\tfrac{d}{2}} &, d-s<0
			\end{cases}.
		\end{equation*}
		 }
	\end{itemize}
\end{lemma}
\begin{proof}
	\ref{itm:a} For $\xi \in 2\, N e_1+\tfrac{1}{2}Q$, $\eta\in Ne_1+Q$ and $-\xi+\eta\in -N e_1 +Q$ the wave phase of the second term can be estimated by
	\begin{align*}
		\abs{\phi^{(2)}(\xi, \eta)}= \abs{\abs{\xi}-\abs{\eta}^2+\abs{\eta-\xi}^2}\lesssim N \, A= \frac{N^2}{\log(N)}.
	\end{align*}
	Therefore, for times $0<T\ll N^{-2} \log(N)$ the real part of the time integral is bounded from below, i.e.
	\begin{equation*}
		\Re \int\limits_0^T e^{is \phi^{(2)}(\xi, \eta)} \di s\gtrsim T.
	\end{equation*}
	Because $\xi \in 2\, N e_1+\tfrac{1}{2}Q$ and $\eta\in Ne_1+\tfrac{1}{2} Q$ implies $-\xi+\eta\in -N e_1 +Q$, we estimate
	\begin{align*}
		\abs{\wop{\xi}^{l}\widehat{A_2^{(2)}}(T, \xi)} &\ge r^2 \abs{Q}^{-1} \Re \int\limits_{Ne_1 +\tfrac{1}{2}Q}  \frac{\wop{\xi}^l \abs{\xi}}{\wop{\eta}^s\wop{\eta-\xi}^s} \int\limits_0^T e^{is\phi^{(2)}(\xi, \eta)} \di s \di \eta\\
		&\gtrsim r^2 T  \fint\limits_{Ne_1 +\tfrac{1}{2}Q} \frac{\wop{\xi}^l \abs{\xi}}{\wop{\eta}^s\wop{\eta-\xi}^s} \di \eta\gtrsim r^2 T N^{-2\,s}  \wop{\xi}^l\abs{\xi}.
	\end{align*}
	We integrating this over $2\, N e_1+\tfrac{1}{2}Q$. This yields the desired bound.
	\begin{equation*}
		\norm{A_2^{(2)}(T)}_{H^l}\gtrsim r^2 T N^{-2\,s}  \norm{\wop{\cdot}^l\abs{\cdot}}_{L^2(2\, N e_1+\tfrac{1}{2}Q)}\gtrsim r^2 T  N^{l-2\,s+1} A^{\tfrac{d}{2}} .
	\end{equation*}

	\ref{itm:b} For $\xi \in (2N+1)e_1+\tfrac{1}{2}Q$, $\eta \in (N+1)e_1 +Q$ and $-\xi+\eta\in -Ne_1+Q$ the wave phase of the second term is bounded $\abs{\phi^{(2)}(\xi, \eta)}\lesssim 1$. Thus, for times $0<T\ll 1$ again the real part of the time integral satisfies $\Re\int_0^T e^{is\phi^{(2)}(\xi, \eta)}\di s \gtrsim T$. Since $\xi \in (2N+1)e_1+\tfrac{1}{2}Q$ and $\eta \in (N+1)e_1 +\tfrac{1}{2}Q$ implies $-\xi+\eta\in -Ne_1+Q$, we estimate
	\begin{align*}
		\abs{\wop{\xi}^l \widehat{A_2^{(2)}}(T,\xi)}\gtrsim r^2 T \fint\limits_{Ne_1 +\tfrac{1}{2}Q} \frac{\wop{\xi}^l \abs{\xi}}{\wop{\eta}^s\wop{\eta-\xi}^s} \di \eta\gtrsim r^2 T N^{-2\,s}\wop{\xi}^l \abs{\xi}.
	\end{align*}
	Integrating this over $(2N+1)e_1+\tfrac{1}{2}Q$ yields the desired bound.
	
	\ref{itm:c} For $\xi \in - Ne_1+\tfrac{1}{2}Q$, $\eta\in Q$ and $-\xi+\eta \in Ne_1+Q$ the wave phase of the second term is of order $\abs{\phi^{(2)}(\xi, \eta) }\simeq N^2$. Again $\Re\int_0^T e^{is\phi^{(2)}(\xi, \eta)}\di s \gtrsim T$ for times $0<T\ll N^{-2}$. Since $\xi - Ne_1+\tfrac{1}{2}Q$ and $\eta\in \tfrac{1}{2}Q$ implies $-\xi+\eta \in Ne_1+Q$, we estimate
	\begin{align*}
		\abs{\wop{\xi}^l \widehat{A_2^{(2)}}(T,\xi)}\gtrsim r^2 T \fint\limits_{\tfrac{1}{2}Q} \frac{\wop{\xi}^l \abs{\xi}}{\wop{\eta}^s\wop{\eta-\xi}^s} \di \eta\gtrsim r^2 T N^{-s} \wop{\xi}^l \abs{\xi}.
	\end{align*}
	Integrating over $- Ne_1+\tfrac{1}{2}Q$ yields the desired bound. 
	
	\ref{itm:c1} For the second bound notice that $\abs{\phi^{(2)}(\xi, \eta)}\lesssim N$ for $\xi \in \tfrac{1}{2}Q$, $\eta \in Ne_1+Q$ and $-\xi+\eta \in Ne_1+Q$. Thus, for times $0<T\ll N^{-1}$ and $\xi \in \tfrac{1}{2}Q$ we calculate
	\begin{align*}
		\abs{\wop{\xi}^l \widehat{A_2^{(2)}}(T,\xi)}\gtrsim r^2 T \fint\limits_{\tfrac{1}{2}Q} \frac{\wop{\xi}^l \abs{\xi}}{\wop{\eta}^s\wop{\eta-\xi}^s} \di \eta\gtrsim r^2 T N^{-2\,s} \wop{\xi}^l  \abs{\xi}.
	\end{align*}
 	We conclude the second bound by integrating this over $ \tfrac{1}{2}Q$.
 	
 	\ref{itm:d} For $\xi \in Ne_1+\tfrac{1}{2}Q$, $\eta\in Q$ and $\xi-\eta \in Ne_1+Q$ the Schrödinger phases of the second term are bounded by 
 	\begin{equation*}
 		\abs{\phi_{\pm}^{(1)}}= \abs{\abs{\xi}^2-\abs{\eta}^2\pm \abs{\xi-\eta}}\lesssim N^{-2}.
 	\end{equation*}
 	Therefore, for times $0<T\ll N^{-2} $ the real part of the time integral is bounded by
 		\begin{equation*}
 			\Re \int\limits_0^T e^{is \phi_{\pm}^{(1)}(\xi, \eta)} \di s\gtrsim T.
 		\end{equation*} 
 	Notice that $\xi-\eta \in Ne_1+Q$ for $\xi \in Ne_1+\tfrac{1}{2}Q$, $\eta\in \tfrac{1}{2}Q$. Thus, we estimate 
 	\begin{align*}
 		\abs{\wop{\xi}^s \widehat{A_2^{(1)}}(T,\xi)}&\ge r^2 \abs{Q}^{-1} \Re \int\limits_{\tfrac{1}{2} Q} \frac{\wop{\xi}^s}{\wop{\eta}^{s}\wop{\xi-\eta}^l} \int\limits_0^T e^{is \phi_{\pm}^{(1)}(\xi,\eta)} \di s \di \eta\\
 		&\gtrsim r^2 T \fint\limits_{\tfrac{1}{2} Q} \frac{\wop{\xi}^s}{\wop{\eta}^{s}\wop{\xi-\eta}^l} \di \eta \gtrsim r^2 T N^{s-l}.
 	\end{align*}
 	Integrating over $Ne_1+\tfrac{1}{2}Q$ yields $\norm{A_2^{(1)}}_{H^s} \gtrsim r^2 T N^{s-l}$.
 	
 	\ref{itm:e} If $\xi \in \tfrac{1}{2}Q$, $\eta \in Ne_1 +Q$ and $\xi-\eta \in -Ne_1 +Q$, then $\abs{\phi_{\pm}^{(1)}}\lesssim N^{-2}$. Thus, for times $0<T\ll N^{-2}$
 	we have $	\Re \int\limits_0^T e^{is \phi_{\pm}^{(1)}(\xi, \eta)} \di s\gtrsim T$. If we further restrict $\xi \in \tfrac{1}{2}Q$, $\eta \in Ne_1 +\tfrac{1}{2}Q$, then $\xi-\eta \in -Ne_1 +Q$. Therefore,
 	\begin{equation*}
 		\abs{\wop{\xi}^s \widehat{A_2^{(1)}}(T,\xi)}\gtrsim r^2 T \fint\limits_{\tfrac{1}{2} Q} \frac{\wop{\xi}^s}{\wop{\eta}^{s}\wop{\xi-\eta}^l} \di \eta \gtrsim r^2 T N^{s-l} \fint\limits_{\tfrac{1}{2}Q} \wop{\eta}^{-s}\di \eta = r^2 T N^{-s-l} \wop{\xi}^s.
 	\end{equation*}
 	Thereby, integrating over $\tfrac{1}{2}Q$ yields 
 	\begin{align*}
 		\norm{A_2^{(1)}}_{H^s} \gtrsim r^2 T N^{-s-l} \norm{\wop{\cdot}^s}_{L^2(\tfrac{1}{2}Q)}= r^2 T N^{-s-l} \begin{cases}
 			A^{\tfrac{d}{2}+s} &, \tfrac{d}{2}+s>0\\
 			\log(A)&, \tfrac{d}{2}+s=0\\
 			1 &, \tfrac{d}{2}+s<0
 		\end{cases}.
 	\end{align*}
 
 	\ref{itm:f} For $\xi \in Ne_1+\tfrac{1}{2}Q$, $\eta\in Q$ and $\xi-\eta \in Ne_1+Q$ the Schrödinger phases of the second term are bounded by 
 	\begin{equation*}
 		\abs{\phi_{\pm}^{(1)}}= \abs{\abs{\xi}^2-\abs{\eta}^2\pm \abs{\xi-\eta}}\lesssim N^{-2}.
 	\end{equation*}
 	Thus, for times $0<T\ll N^{-2}$
 	\begin{equation*}
 		\Re \int\limits_0^T e^{is \phi_{\pm}^{(1)}(\xi, \eta)} \di s\gtrsim T.
 	\end{equation*} 
 	If $\xi \in Ne_1+\tfrac{1}{2}Q$, $\eta\in \tfrac{1}{2}Q$, then $\xi-\eta \in Ne_1+Q$. Therefore,
 	\begin{align*}
 		\abs{\wop{\xi}^s \widehat{A_2^{(1)}}(T,\xi)}&\ge r^2 \abs{Q}^{-1} \Re \int\limits_{\tfrac{1}{2} Q} \frac{\wop{\xi}^s}{\wop{\eta}^{l}\wop{\xi-\eta}^s} \int\limits_0^T e^{is \phi_{\pm}^{(1)}(\xi,\eta)} \di s \di \eta \gtrsim r^2 T \fint\limits_{\tfrac{1}{2} Q} \frac{\wop{\xi}^s}{\wop{\eta}^{s}\wop{\xi-\eta}^l} \di \eta\\
 		&\gtrsim r^2 T N^{s-l} \fint\limits_{\tfrac{1}{2}Q} \wop{\eta}^{-s}\di \eta = r^2 T N^{s-l} \begin{cases}
 			A^{-s} &, d-s >0\\
 			A^{-d}\log (A) &,d-s=0\\
 			A^{-d} &, d-s<0
 		\end{cases}.
 	\end{align*}
 	Integrating over $Ne_1+\tfrac{1}{2}Q$ yields the claim. 
\end{proof}
Combining all results proven in this section on the series expansion for certain initial data enables us to prove \cref{Theorem 1.1}. 
\subsection{Proof of Theorem 1.1}\label{sec:proofth} Let $d\in \BN$ and $(s,l)\in \BR^2$. We split the proof of \cref{Theorem 1.1} into several parts according to the cases of initial data established in \cref{lem:lower_sobo_bounds}. Recall that all calculations are based on a big $N\in \BN$, which we send to $\infty$ later. By the choice of initial data from \eqref{eq:anfangsdaten2}, we have in any case $\norm{u_0}_{H^s}, \norm{n_0}_{H^l}\lesssim r= (\log^2 N)^{-1}\to 0$ as $N\to \infty$. For the remaining proof we define $\rho$ as in \cref{lem:existenz} and the constant $C_{s,l,d}:= C\,\sqrt{2}^{\,d+7+2\abs{l} +2s_\star}$, where $C>0$ is the constant from \cref{lem:Sobolevbounds} and $s_\star$ as defined in \cref{lem:existenz}. In the end these constant will play no role for the proof of norm inflation, because they don't scale with $N$.\medskip

\textbf{(a)} For initial data as in \cref{lem:lower_sobo_bounds} \ref{itm:a} the Zakharov system exhibits norm inflation in the wave coordinate for any Sobolev regularities satisfying $2s-l<\tfrac{d-2}{2}$, $s-l<\tfrac{1}{2}$ and $l\ge-\tfrac{d+2}{2}$. Set \begin{equation*}
	T:= r^{-3} (\log N)^{\tfrac{d}{2}} N^{2\,s-l-1-\tfrac{d}{2}} .
\end{equation*} Since $2\,s-l-\tfrac{d-2}{2}<0$, we conclude by this choice of $T>0$ that $T\ll N^{-2}(\log N)$. Thus by \cref{lem:lower_sobo_bounds}, the wave coordinate of the second term in $H^l$-norm is bounded from below
\begin{equation*}
	\norm{A_2^{(2)}(T)}_{H^l}\gtrsim r^2 T  N^{l-2\,s+1} A^{\tfrac{d}{2}}= r^{-1}=\log^2 N \to \infty \text{ as } N \to \infty.
\end{equation*}
Next, we want to ensure that the series $\sum_{n=1}^{\infty}A_n \icol{u_0\\ 0}$ converges, is a mild solution of the reduced Zakharov system and that its wave coordinate enjoys the same lower bound as that of the second term. To achieve this, we begin by calculating the $L^2$-norm of the initial data. We have $\norm{n_0}_{L^2}=0$ and 
\begin{align*}
	\norm{u_0}_{L^2}= r\abs{Q}^{-\tfrac{1}{2}} \norm{\wop{\cdot}^{-s}}_{L^2(\pm Ne_1+Q)}\simeq r N^{-s}.
\end{align*}
Therefore, the series converges absolutely and is a solution by \cref{lem:existenz}, since we assume $s-l<\tfrac{1}{2}$ and thus
\begin{align*}
	\rho = C_{s,l,d} \abs{Q}^{\tfrac{1}{2}} N^{\tfrac{1}{2}} \norm{u_0}_{L^2} T \simeq  r^{-2} N^{s-l-\tfrac{1}{2}}\ll 1.
\end{align*}
By \cref{lem:besselpot_estimates}, 
\begin{equation*}
	\norm{\wop{\cdot}^{l} \, \abs{\cdot}}_{L^2(\Omega_{n}^{(2)}) } \le 3^2 \big(2^{\tfrac{d+5}{2}+\abs{l}+1}\big)^n \norm{\wop{\cdot}^{l}\,\abs{\cdot}}_{L^2(Q\cup (Ne_1+Q))}.
\end{equation*}
For $l+1\ge 0$ this can be simply bounded by 
\begin{equation*}
	\norm{\wop{\cdot}^{l} \, \abs{\cdot}}_{L^2(\Omega_{n}^{(2)}) } \le 3^2 \big(2^{\tfrac{d+5}{2}+\abs{l}+1}\big)^n N^{l+1} \abs{Q}^{\tfrac{1}{2}}
\end{equation*}
and for $l+1<0$ we have
\begin{equation*}
	\norm{\wop{\cdot}^{l} \, \abs{\cdot}}_{L^2(\Omega_{n}^{(2)}) } \le 2\,3^2 \big(2^{\tfrac{d+5}{2}+\abs{l}+1}\big)^n \norm{\wop{\cdot}^{l}\,\abs{\cdot}}_{L^2(Q)}\lesssim \big(\frac{C_{s,l,d}}{C}\big)^n \begin{cases}
		A^{\tfrac{d+2}{2}+l} &, \frac{d+2}{2} + l>0\\
		\log(A)&, \frac{d+2}{2}+l=0
	\end{cases}.
\end{equation*}
The wave coordinate $A_n^{(2)}$ is zero if $n\ge 2$ is uneven since the wave initial data is zero. We use \cref{lem:Sobolevbounds} and the above calculations on the Bessel potential to estimate the wave coordinate of the solution without the first two terms $\sum_{n=3}^\infty A_{n}^{(2)}$ by
\begin{align*}
	\sum_{n=3}^{\infty}\norm{A_{n}^{(2)}(T)}_{H^l} &\lesssim  \abs{Q}^{-\tfrac{1}{2}}   N^{-\tfrac{1}{2}}\norm{\hat{u}_0}_{L^2}  \rho^{3}\begin{cases}
		N^{l+1} A^{\tfrac{d}{2}} &, l+1\ge 0\\
		A^{\tfrac{d+2}{2}+l} &, l+1< 0,\frac{d+2}{2} + l>0\\
		\log(A)&, \frac{d+2}{2}+l=0
	\end{cases}\\
&\lesssim 	 (\log^2 N)^{3}\,\begin{cases}
 N^{2\,s-2l-1}  &, l+1\ge 0\\
		 (\log N)^{-l-1} N^{2\,s-2l-1} &, l+1< 0,\frac{d+2}{2} + l>0\\
	(\log N)^{\tfrac{d+2}{2}} N^{2\,s+d+1} &, \frac{d+2}{2}+l=0
\end{cases}\quad \ll \log^2 N\lesssim 	\norm{A_2^{(2)}(T)}_{H^l}.
\end{align*}
Therefore, we found a sequence (in $N\in \BN$) of times $T$ and initial data $(u_0, 0)\in H^{\infty, \infty}$ such that the corresponding solutions to the reduced Zakharov system  $\sum_{n=1}^\infty A_n \icol{u_0\\ 0}$ exist at least up to time $T$, the initial data as well as the time $T$ converge to zero and the solution diverges
\begin{align*}
	\norm{\sum_{n=1}^\infty A_n^{(2)}(T)}_{H^l}\ge \norm{ A_2^{(2)}(T)}_{H^l}-\sum_{n=3}^\infty\norm{ A_n^{(2)}(T)}_{H^l}\gtrsim \log^2 N \to \infty \text{ as } N\to \infty.
\end{align*}

\textbf{(b)} For initial data as in \cref{lem:lower_sobo_bounds} \ref{itm:b} the Zakharov system exhibits norm inflation in the wave coordinate for any Sobolev regularities satisfying $2s-l<\tfrac{1}{2}$, $s-l<\tfrac{1}{2}$ and $l\ge-1$. We pick
\begin{equation*}
	T:= r^{-3} N^{2s-l-\tfrac{1}{2}}.
\end{equation*}
Because $2s-l<\tfrac{1}{2}$, we have $T\ll 1$ and thus by \cref{lem:lower_sobo_bounds}
\begin{equation*}
	\norm{A_2^{(2)}(T)}_{H^l}\gtrsim \log^2 N \to \infty \text{ as } N\to \infty.
\end{equation*}
Additionally, $\norm{n_0}_{L^2}=0$ and $\norm{u_0}_{L^2}\simeq rN^{-s}$. By \cref{lem:existenz}, the assumption $s-l<\tfrac{1}{2}$ and 
\begin{equation*}
	\rho = C_{s,l,d} \abs{Q}^{\tfrac{1}{2}} N^{\tfrac{1}{2}} \norm{u_0}_{L^2} T \simeq  r^{-2} N^{s-l-\tfrac{1}{2}}\ll 1
\end{equation*}
the series $\sum_{n=1}^{\infty} A_n\icol{u_0\\ 0}$ converges in $H^{s,l}$ and is a solution to the reduced Zakharov system. By \cref{lem:besselpot_estimates}, we know for $l+1\ge0$
\begin{equation*}
	\norm{\wop{\cdot}^{l} \, \abs{\cdot}}_{L^2(\Omega_{n}^{(2)}) } \le 3^2 \big(2^{\tfrac{d+5}{2}+\abs{l}+1}\big)^n \norm{\wop{\cdot}^{l}\,\abs{\cdot}}_{L^2(Q\cup (Ne_1+Q))}\lesssim \big(\frac{C_{s,l,d}}{C}\big)^n  N^{l+\tfrac{1}{2}}.
\end{equation*}
Together with \cref{lem:Sobolevbounds} we estimate
\begin{equation*}
	\sum_{n=3}^{\infty}\norm{A_{n}^{(2)}(T)}_{H^l} \lesssim  \abs{Q}^{-\tfrac{1}{2}}   N^{-\tfrac{1}{2}}\norm{\hat{u}_0}_{L^2}  \rho^{3} N^{l+\tfrac{1}{2}}\lesssim  (\log^2 N)^{5} N^{2s-2l-1} \ll  \log^2 N \lesssim \norm{A_2^{(2)}(T)}_{H^l}.
\end{equation*}
By the same reasoning as in (a), we conclude norm inflation.\medskip

\textbf{(c)} For initial data as in \cref{lem:lower_sobo_bounds} \ref{itm:c} the Zakharov system exhibits norm inflation in the wave coordinate for Sobolev regularities satisfying $s-l<-1$ and $s\ge0$. We set
\begin{equation*}
	T:= r^{-3} N^{s-l-1}
\end{equation*}
and, thus, $T\ll N^{-2}$, because we assumed $s-l<-1$. By \cref{lem:lower_sobo_bounds}, 
\begin{equation*}
	\norm{A_2^{(2)}(T)}_{H^l}\gtrsim \log^2 N \to \infty \text{ as } N\to \infty.
\end{equation*}
Our initial data in $L^2$ are of size $\norm{n_0}_{L^2}=0$, $\norm{u_0}_{L^2(Q)}\simeq r$ and $\norm{u_0}_{L^2(Q^c)}\simeq rN^{-s}$. Thus, $\norm{u_0}_{L^2}\simeq r$ by $s\ge 0$. 
\begin{equation*}
		\rho = C_{s,l,d} \abs{Q}^{\tfrac{1}{2}} N^{\tfrac{1}{2}} \norm{u_0}_{L^2} T \simeq r^{-2} N^{s-l-\tfrac{1}{2}}\ll 1
\end{equation*}
implies via \cref{lem:existenz} the convergence $\sum_{n=1}^{\infty}A_n$ and it solves the reduced Zakharov system. Since we assumed $l+1\ge0$, we estimate using \cref{lem:besselpot_estimates}  
\begin{align*}
	\norm{\wop{\cdot}^{l} \, \abs{\cdot}}_{L^2(\Omega_{n}^{(2)}\setminus nQ)} &\le 3^2 \big(2^{\tfrac{d+5}{2}+\abs{l}+1}\big)^n \norm{\wop{\cdot}^{l}\,\abs{\cdot}}_{L^2(Q\cup (Ne_1+Q))}\lesssim \big(\frac{C_{s,l,d}}{C}\big)^n  N^{l+1}\\
	\intertext{and in a similar manner}
	\norm{\wop{\cdot}^{l} \, \abs{\cdot}}_{L^2(nQ)} &\lesssim \big(\frac{C_{s,l,d}}{C}\big)^n .
\end{align*}
An application of \cref{lem:Sobolevbounds} together with the assumption $s-l<-1$ yields
\begin{align*}
	\sum_{n=3}^{\infty} \norm{A_{n}^{(2)}(T)}_{H^l} &\lesssim  \sum_{n=2}^\infty C^{2n} \abs{Q}^{n-1}T^{2n-1} \,   N^{n-1}\norm{\hat{u}_0}_{L^2}^{2n} \Bigg(  \norm{\wop{\cdot}^l\abs{\cdot}}_{L^2(nQ)} + \norm{\wop{\cdot}^l\abs{\cdot}}_{L^2(\Omega_{n}^{(2)}\setminus nQ)}\frac{  \norm{\hat{u}_0}_{L^2(nQ^c)} }{ \norm{\hat{u}_0}_{L^2}}  \Bigg)\\
	&\lesssim   N^{-\tfrac{1}{2}}\norm{\hat{u}_0}_{L^2} \rho^{3} (  1+ N^{l-s+1} )\simeq  (\log^2 N)^5    N^{3s-3l-2} (  1+ N^{l-s+1} )\ll \log^2 N \lesssim 	\norm{A_2^{(2)}(T)}_{H^l}.
\end{align*}
Again the same reason as in (a) implies norm inflation.\medskip

\textbf{(c')} We use the initial data from \cref{lem:lower_sobo_bounds} \ref{itm:c1} just as in case (c). The Zakharov system exhibits norm inflation in the wave coordinate for $(s,l)\in \BR^2$ such that $l\le -1$ and $s<-1$. We choose
\begin{equation*}
	T= r^{-3}\,N^{2\,s}.
\end{equation*} 
The assumption $T\ll N^{-1}$ and \cref{lem:lower_sobo_bounds} \ref{itm:c1} yield
\begin{equation*}
	\norm{A_2^{(2)}(T)}_{H^l}\gtrsim \log^2 N \to \infty \text{ as }N \to \infty.
\end{equation*}
In $L^2$ the initial data are of size $\norm{u_0}_{L^2(Q)}\simeq r$, $\norm{u_0}_{L^2(Q^c)}\simeq rN^{-s}$ and, therefore, $\norm{u_0}_{L^2}\simeq rN^{-s}$ by $s< 0$. 
\begin{equation*}
	\rho = C_{s,l,d} \abs{Q}^{\tfrac{1}{2}} N^{\tfrac{1}{2}} \norm{u_0}_{L^2} T \simeq r^{-2}  N^{s-\tfrac{1}{2}}   \ll 1
\end{equation*}
implies the convergence of the solution by \cref{lem:existenz}. Since we assumed $l+1\le 0$,
\begin{equation*}
	\norm{\wop{\cdot}^l\abs{\cdot}}_{L^2(\Omega_{n}^{(2)})}\lesssim \big(\frac{C_{s,l,d}}{C}\big)^n.
\end{equation*}
The same Sobolev bounds on the terms $A_n^{(1)}(T)$ as in (c) hold, because $l+1\le 0$. Thereby, 
\begin{equation*}
	\sum_{n=3}^{\infty} \norm{A_{n}^{(2)}(T)}_{H^l} \lesssim   N^{-\tfrac{1}{2}}\norm{\hat{u}_0}_{L^2} \rho^{3} \simeq  (\log^2 N)^5  N^{2s-2}  \ll\log^2 N \lesssim \norm{A_2^{(2)}(T)}_{H^l}.
\end{equation*}
This again implies norm inflation.\medskip

\textbf{(d)} We consider initial data as in \cref{lem:lower_sobo_bounds} \ref{itm:d}. The Zakharov system exhibits norm inflation in the Schrödinger coordinate for regularities satisfying $l-s<-2$ and $s>0$. We pick
\begin{equation*}
	T=r^{-3}\,N^{l-s}.
\end{equation*}
Thus, the assumption $T\ll N^{-2}$ and \cref{lem:lower_sobo_bounds} yield
\begin{equation*}
	\norm{A_2^{(1)}(T)}_{H^s}\gtrsim \log^2 N \to \infty \text{ as }N \to \infty. 
\end{equation*}
The initial data satisfy $\norm{u_0}_{L^2}\simeq r$ and $\norm{n_0}_{L^2}\simeq r N^{-l}$. Thereby, \cref{lem:existenz} with 
\begin{equation*}
	\rho = C\,2^{\tfrac{d+7}{2}+\abs{l} +s_\star}\abs{Q}^{\tfrac{1}{2}} \max\{ N^{\tfrac{1}{2}}\, \norm{\hat{u}_0}_{L^2}, \norm{\hat{n}_0}_{L^2} \} \, T \simeq  r^{-2} N^{\tfrac{1}{2}\vee (-l)} \, N^{l-s}=r^{-2} \begin{cases}
	  N^{l-s+\tfrac{1}{2}} &, l\ge -\tfrac{1}{2}\\
	  N^{-s} &, l< -\tfrac{1}{2}
	\end{cases}\ll 1
\end{equation*}
implies the convergence of the series $\sum_{n=1}^{\infty}A_n$ and that it solves the reduced Zakharov system. Similar arguments as in \cref{lem:besselpot_estimates} yield 
\begin{align*}
	\norm{\wop{\cdot}^{s}}_{L^2(\Omega_{n}^{(1)} \setminus nQ) } &\le3^2 \big(2^{\tfrac{d+5}{2}+\abs{s}}\big)^n \norm{\wop{\cdot}^{s}}_{L^2(Ne_1+Q)}\lesssim \big(\frac{C_{s,l,d}}{C}\big)^n N^s\\
	\intertext{and}
	\norm{\wop{\cdot}^{s}}_{L^2(nQ)} &\lesssim \big(\frac{C_{s,l,d}}{C}\big)^n\norm{\wop{\cdot}^{s}}_{L^2(Q)} \lesssim \big(\frac{C_{s,l,d}}{C}\big)^n.
\end{align*}
Since $\max\{N^{\tfrac{1}{2}} \norm{\hat{u}_0}_{L^2},\norm{\hat{n}_0}_{L^2} \}=r N^{\tfrac{1}{2}\vee (-l)}$, we apply \cref{lem:Sobolevbounds} to receive for $l\ge -\tfrac{1}{2}$ 
\begin{align*}
	\norm{A_n^{(1)}(T)}_{H^s} &\le C^n\abs{Q}^{\tfrac{n}{2}-1}T^{n-1} \, \max\limits_{\substack{1\le k\le n\\ 2\nmid k }}\Big\{ N^{\tfrac{k-1}{2}}(\norm{\hat{u}_0}_{L^2})^k (\norm{\hat{n}_0}_{L^2})^{n-k} \Big\}\\
	&\qquad\qquad\qquad\qquad\qquad\qquad\qquad\qquad\cdot\Big(  \norm{\wop{\cdot}^s}_{L^2(nQ)} + \norm{\wop{\cdot}^s}_{L^2(\Omega_{n}^{(1)}\setminus nQ)}\, N^{-l-\tfrac{1}{2}}  \Big)\\
	&= \begin{cases}
		C^n\abs{Q}^{\tfrac{n}{2}-1}T^{n-1} \,N^{\tfrac{n-1}{2}}(\norm{\hat{u}_0}_{L^2})^n\Big(  \norm{\wop{\cdot}^s}_{L^2(nQ)} + \norm{\wop{\cdot}^s}_{L^2(\Omega_{n}^{(1)}\setminus nQ)}\, N^{-l-\tfrac{1}{2}}  \Big)	&, 2\nmid n\\
		C^n\abs{Q}^{\tfrac{n}{2}-1}T^{n-1} N^{\tfrac{n-2}{2}}(\norm{\hat{u}_0}_{L^2})^{n-1} \norm{\hat{n}_0}_{L^2} \Big(  \norm{\wop{\cdot}^s}_{L^2(nQ)} + \norm{\wop{\cdot}^s}_{L^2(\Omega_{n}^{(1)}\setminus nQ)}\, N^{-l-\tfrac{1}{2}}  \Big)	&, 2\mid n
	\end{cases}\\
	&\lesssim \begin{cases}
		r \rho^{n-1}\big(  1 +  N^{s-l-\tfrac{1}{2}}  \big)	&, 2\nmid n\\
		  rN^{-l-\tfrac{1}{2}} \rho^{n-1}\big(  1 +  N^{s-l-\tfrac{1}{2}}  \big)	&, 2\mid n
	\end{cases}.
\end{align*}
Thus, we conclude norm inflation in case $l\ge -\tfrac{1}{2} $ by
\begin{align*}
	\sum\limits_{n=3}^\infty \norm{A_n^{(1)}(T)}_{T} &\lesssim r \rho^{2}\big(  1 +  N^{s-l-\tfrac{1}{2}}  \big)+rN^{-l-\tfrac{1}{2}} \rho^{3}\big(  1 +  N^{s-l-\tfrac{1}{2}}  \big)\\
	&\simeq r^{-3}     N^{2l-2s+1}  + r^{-3}    N^{l-s+\tfrac{1}{2}}     +r^{-5}   N^{2l-3s+1}+r^{-5}   N^{l-2s+\tfrac{1}{2}}\ll \log^2 N \lesssim \norm{A_2^{(1)}(T)}_{H^s}.
\end{align*}
The last line is true, because we assumed $l-s<-2$ and $s>0$. Thus, norm inflation in case $l\ge -\tfrac{1}{2}$ follows. For $l<-\tfrac{1}{2}$ 
\begin{equation*}
	\norm{A_n^{(1)}(T)}_{H^s}\le C^n\abs{Q}^{\tfrac{n}{2}-1}T^{n-1} \, \norm{\hat{u}_0}_{L^2} (\norm{\hat{n}_0}_{L^2})^{n-1} \Big(  \norm{\wop{\cdot}^s}_{L^2(nQ)} + \norm{\wop{\cdot}^s}_{L^2(\Omega_{n}^{(1)}\setminus nQ)}\Big)\lesssim r N^s\rho^{n-1}.
\end{equation*}
Therefore, we conclude norm inflation by 
\begin{align*}
\sum\limits_{n=3}^\infty \norm{A_n^{(1)}(T)}_{T} \lesssim r N^s\rho^{2} \simeq  (\log^2 N)^3 N^{-s}\ll \log^2 N\lesssim	\norm{A_2^{(1)}(T)}_{H^s}.
\end{align*}

\textbf{(e)} We pick initial data as in \cref{lem:lower_sobo_bounds} \ref{itm:e}. The Zakharov system exhibits norm inflation for regularities $s<0$, $l-s\le -\tfrac{1}{2}$ and $l<\tfrac{d}{2}-2$ for $s\ge -\tfrac{d}{2}$ and $l+s<-2$ for $s< -\tfrac{d}{2}$. We choose
\begin{equation*}
	T:= r^{-3}  \begin{cases}
		N^{l-\tfrac{d}{2}} (\log N)^{\tfrac{d}{2}+s} &, \tfrac{d}{2}+s>0\\
		N^{s+l}\,(\log A)^{-1}&, \tfrac{d}{2}+s=0\\
		N^{s+l} &, \tfrac{d}{2}+s<0
	\end{cases}.
\end{equation*}
By the above assumptions on $(s,l)$, $T\ll N^{-2}$. Thus, by \cref{lem:lower_sobo_bounds} 
\begin{equation*}
		\norm{A_2^{(1)}(T)}_{H^s}\gtrsim \log^2(N)\to \infty \text{ as } N \to \infty. 
\end{equation*}
Now we calculate the $L^2$-norm of the initial data. $\norm{u_0}_{L^2}\simeq rN^{-s}$ and $\norm{n_0}_{L^2} \simeq r N^{-l}$. Therefore, \begin{equation*}
	\max \{ N^{\tfrac{1}{2}}  \norm{u_0}_{L^2}, \norm{n_0}_{L^2} \} \simeq r N^{-l}
\end{equation*} for $l-s\le -\tfrac{1}{2}$. Additionally,
\begin{equation*}
	\rho = C\,2^{\tfrac{d+7}{2}+\abs{l} +s_\star}\abs{Q}^{\tfrac{1}{2}} \max\{ N^{\tfrac{1}{2}}\, \norm{\hat{u}_0}_{L^2}, \norm{\hat{n}_0}_{L^2} \} \, T \simeq   r^{-2}  \begin{cases}
		(\log N)^{s} &, \tfrac{d}{2}+s>0\\
		(\log N)^{-\tfrac{d+2}{2}}&, \tfrac{d}{2}+s=0\\
		N^{s+\tfrac{d}{2}} (\log N)^{-\tfrac{d}{2}} &, \tfrac{d}{2}+s<0
	\end{cases}\quad\ll 1.
\end{equation*}
An application of \cref{lem:existenz} yields the existence of the solution $\sum_{n=1}^{\infty} A_{n}$. Because we assumed $s<0$ and by \cref{lem:besselpot_estimates},
\begin{equation*}
	\norm{\wop{\cdot}^s}_{L^2(\Omega_n^{(1)})}\lesssim \big(\frac{C_{s,l,d}}{C}\big)^n \begin{cases}
		A^{\tfrac{d}{2}+s} &, \tfrac{d}{2}+s>0\\
		\log A &, \tfrac{d}{2}+s=0\\
		1 &, \tfrac{d}{2}+s<0
	\end{cases}.
\end{equation*}
Additionally, \cref{lem:Sobolevbounds} yields
\begin{align*}
	\norm{A_n^{(1)}(T)}_{H^s}&\le \norm{\wop{\cdot}^s}_{L^2(\Omega_n^{(1)})} C^n\abs{Q}^{\tfrac{n}{2}-1}T^{n-1} \, \max\limits_{\substack{1\le k\le n\\ 2\nmid k }}\Big\{ N^{\tfrac{k-1}{2}}(\norm{\hat{u}_0}_{L^2})^k (\norm{\hat{n}_0}_{L^2})^{n-k} \Big\}\\
	&\simeq   \norm{\hat{u}_0}_{L^2}\rho^{n-1} \begin{cases}
		A^{s} &, \tfrac{d}{2}+s>0\\
		A^{-\tfrac{d}{2}} \log A &, \tfrac{d}{2}+s=0\\
		A^{-\tfrac{d}{2}}&, \tfrac{d}{2}+s<0
	\end{cases}.
\end{align*}
Therefore, 
\begin{align*}
	\sum_{n=3}^{\infty} \norm{A_n^{(1)}(T)}_{H^s}&\lesssim \norm{\hat{u}_0}_{L^2}\rho^{2} \begin{cases}
		A^{s} &, \tfrac{d}{2}+s>0\\
		A^{-\tfrac{d}{2}} \log A &, \tfrac{d}{2}+s=0\\
		A^{-\tfrac{d}{2}}&, \tfrac{d}{2}+s<0
	\end{cases}\simeq   (\log^2 N)^3  \begin{cases}
	 (\log N)^{s} &, \tfrac{d}{2}+s>0\\
	(\log A)^{-\tfrac{d}{2}-1} &, \tfrac{d}{2}+s=0\\
	N^{\tfrac{d}{2}+s}(\log N)^{-\tfrac{d}{2}}&, \tfrac{d}{2}+s<0
\end{cases}\\
&\ll \log^2 N \lesssim 	\norm{A_1^{(1)}(T)}_{H^s}
\end{align*}
implies norm inflation.\medskip

\textbf{(f)} Consider initial data as in \cref{lem:lower_sobo_bounds} \ref{itm:f}. The Zakharov system experiences norm inflation for Sobolev regularities $l<\tfrac{d}{2}-2$, $0<s<\tfrac{d}{2}$ and $l-s\le-\tfrac{1}{2}$. We pick 
\begin{equation*}
	T=  r^{-3}(\log N)^{-s+\tfrac{d}{2}} N^{l-\tfrac{d}{2}} 
\end{equation*}
and notice $T\ll N^{-2}$ by the assumptions $l<\tfrac{d}{2}-2$. This yields $\norm{A_2^{(1)}(T)}_{H^s}\gtrsim \log^2 N \to \infty$ as $N\to \infty$ by \cref{lem:lower_sobo_bounds}. The initial data satisfy $\norm{n_0}_{L^2}\simeq r N^{-l}$ and by $\tfrac{d}{2}>s$
\begin{equation*}
	\norm{u_0}_{L^2}\simeq r N^{-s} (\log N)^{s}.
\end{equation*}
Since we assumed $l\le \tfrac{d}{2}-2$ and $l-s\le -\tfrac{1}{2}$, $\max \{ N^{\tfrac{1}{2}}  \norm{u_0}_{L^2}, \norm{n_0}_{L^2} \} =r N^{-l}$. We calculate
\begin{equation*}
	\rho = C\,2^{\tfrac{d+7}{2}+\abs{l} +s_\star}\abs{Q}^{\tfrac{1}{2}} \max\{ N^{\tfrac{1}{2}}\, \norm{\hat{u}_0}_{L^2}, \norm{\hat{n}_0}_{L^2} \} \, T \simeq    r^{-2}  (\log N)^{-s}\ll 1.
\end{equation*}
Therefore, the series $\sum_{n=1}^{\infty}A_n$ converges and is a mild solution of the Zakharov system by \cref{lem:existenz}. Because we assumed $s>0$ and by \cref{lem:besselpot_estimates},
\begin{equation*}
	\norm{\wop{\cdot}^s}_{L^2(\Omega_{n}^{(1)})}\lesssim \big(\frac{C_{s,l,d}}{C}\big)^n A^{\tfrac{d}{2}}N^s.
\end{equation*}
Now by \cref{lem:Sobolevbounds}, 
\begin{equation*}
	\sum_{n=3}^{\infty} \norm{A_n^{(2)}(T)}_{H^s} \lesssim \norm{u_0}_{L^2} N^s \rho^2\simeq  (\log^2 N)^3 (\log N)^{-s} \ll \log^2 N \lesssim \norm{A_2^{(1)}(T)}_{H^s}
\end{equation*}
and we conclude norm inflation.\qed
	\appendix
	\section{}
	The following proposition is a technical estimate on the sequence used in \cref{prop:summe2}.
	\begin{proposition}\label{prop:sequence}
		Define two sequences recursively for fixed starting values $b_1^1,b_1^2 \in [0,1]$
		\begin{align*}
			b_{n}^1:= C_1\,\frac{1}{n-1}\sum\limits_{\substack{n_1,n_2\in \BN\\ n_1+n_2=n}} \underline{m}^{\alpha}b_{n_1}^1b_{n_2}^2,\qquad			b_{n}^2:=C_2\,\frac{n}{n-1}\sum\limits_{\substack{n_1,n_2\in \BN\\ n_1+n_2=n}}\underline{m}^{\alpha} b_{n_1}^1b_{n_2}^1,
		\end{align*}
		where $\underline{m}:= \min\{ n_1,n_2 \}$ depends on the summation indices and $\alpha>0$, $C_1,C_2>0$ are constants. There exists $\gamma> 1$ such that 
		\begin{equation}\label{eq:sequence_induction_start}
			b_n^1 \le n^{-\alpha-2}\gamma^{n-1}, \qquad	b_n^2 \le n^{-\alpha-2}\gamma^{n-1}.
		\end{equation}
	\end{proposition}
	\begin{proof}
		We define $C_{\max}:= \max\{ C_1,C_2 \}$ and choose $\gamma:= 2^{\alpha+5}\, C_{\max}$. Trivially the induction start is true. Furthermore,  
		\begin{equation*}
			b_2^1 =C_1\, b_1^1\,b_1^2 \le C_{\max} \le 2^{-\alpha-2}\gamma\quad \text{and} \quad b_2^2= C_2\,b_1^1\, b_1^1 \le C_{\max} \le 2^{-\alpha-2}\gamma.
		\end{equation*}
		Let's assume the estimates \eqref{eq:sequence_induction_start} for all $k<2n$ with our choice of $\gamma$ for a fixed $n$. This yields by changing the summation order
		\begin{equation*}
			b_{2n}^1\le C_1\, \frac{2}{2n-1}\sum_{k=1}^{n} k^{\alpha} k^{-\alpha-2} (2n-k)^{-\alpha-2} \gamma^{2n-2}\le C_{\max} \frac{2^{\alpha+3}}{2n-1} (2n)^{-\alpha-2}\gamma^{2n-2} \sum_{k=1}^n k^{-2}.
		\end{equation*}
		The second inequality in the previous line follows from $\frac{2n}{2n-k}\le 2$ for $k\in \{1,\dots, n\}$. Therefore,
		\begin{equation*}
			b_{2n}^1\le C_{\max}\underbrace{\frac{\pi^2}{6(2n-1)}}_{\le 2}2^{\alpha+3}(2n)^{-\alpha-2}\gamma^{2n-2}\le \underbrace{\frac{C_{\max}\, 2^{\alpha+4}}{\gamma}}_{\le 1}(2n)^{-\alpha-2}\gamma^{2n-1}.
		\end{equation*}
		In a similar manner, we prove the estimate for the sequence $\{b^2_n\}$. 
		\begin{align*}
			b_{2n}^2&\le 2\,C_2\,\frac{2n}{2n-1}\sum_{k=1}^{n} k^{\alpha} k^{-\alpha-2} (2n-k)^{-\alpha-2} \gamma^{2n-2} \le C_{\max}\,\underbrace{\frac{\pi^2 (2n)}{6(2n-1)}}_{\le 4}2^{\alpha+3}(2n)^{-\alpha-2}\gamma^{2n-2}\\
			&\le \underbrace{\frac{C_{\max}\,2^{\alpha+5}}{\gamma}}_{\le 1}(2n)^{-\alpha-2}\gamma^{2n-1}.
		\end{align*}
		Now we prove the estimate \eqref{eq:sequence_induction_start} for the next uneven index. We proceed just as above.
		\begin{align*}
			b_{2n+1}^1&\le C_1\,\frac{2}{2n}\sum_{k=1}^{n} k^{\alpha} k^{-\alpha-2} (2n+1-k)^{-\alpha-2} \gamma^{2n-1}\le C_{\max}\frac{2^{\alpha+3}}{2n} (2n+1)^{-\alpha-2}\gamma^{2n-1} \sum_{k=1}^n k^{-2}\\
			&\le C_{\max}\underbrace{\frac{\pi^2}{12n}}_{\le 1}2^{\alpha+3}(2n)^{-\alpha-2}\gamma^{2n-1}\le \underbrace{\frac{C_{\max}\,2^{\alpha+3}}{\gamma}}_{\le 1}(2n)^{-\alpha-2}\gamma^{2n},\\
			b_{2n+1}^2&\le C_2\,2\,\frac{2n+1}{2n}\sum_{k=1}^{n} k^{\alpha} k^{-\alpha-2} (2n+1-k)^{-\alpha-2} \gamma^{2n-1} \le  \underbrace{C_{\max}\,\frac{\pi^2 (2n+1)}{12n}\, \gamma \, 2^{\alpha+3}}_{\le 1}(2n+1)^{-\alpha-2}\gamma^{2n}.
		\end{align*}
	\end{proof}
We recall a classical result on point wise products of Sobolev function introduced in \cite{zolesio_1977}. We will use this lemma to prove \cref{lem:contsolopabove}.
\begin{lemma}[{Sobolev product estimate \cite{zolesio_1977}, \cite[Theorem 5.1]{behzadan2017multiplication}}]\label{lem:sobolevproduct}
	Let $s_1,s_2\ge s\ge 0$ and $s_1+s_2>s+\tfrac{d}{2}$ with dimension $d \ge 1$. For any $u\in H^{s_1}(\BR^d)$ and $n\in H^{s_2}(\BR^d)$
	\begin{equation*}
		\norm{un}_{H^s}\lesssim \norm{u}_{H^{s_1}}\norm{n}_{H^{s_2}}.
	\end{equation*}
\end{lemma}
The following lemma is used in \cref{lem:existenz} to legitimize that the power series expansion, $\sum_{n=1}^{\infty}A_n\icol{u_0\\n_0}$, is a mild solution. 
\begin{lemma}[Continuity of $\Phi$ above $(0, -1)$]\label{lem:contsolopabove}
	Let $s_2,l_2+1\ge 0$, $d\in \BN$ and $l_1:= s_2$, $s_1$ such that
	\begin{equation*}
	s_1>\max \{\tfrac{d}{2}, \tfrac{l_2}{2}+\tfrac{d+2}{4}\} \quad \text{and}\quad 
	s_1\ge \max \{ s_2, l_2+1 \} 
	\end{equation*}
	holds. For arbitrary initial data $(u_0,n_0)\in H^{s_2,l_2}$ and any $T,\delta'>0$ the operator defined in \eqref{eq:integralgleichungalternative}
	$$
	\Phi_{\icol{u_0\\ n_0}}:\overline{B_{\delta'}^{C([0,T]; H^{s_1,l_1})}(0)} \to C([0,T];H^{s_2,l_2})
	$$
	is a continuous map.
\end{lemma}
We will choose $\delta'>0$ dependent on the size of our constructed solution in the respective Sobolev norms. 
\begin{proof}
	We fix initial data $(u_0,n_0)\in H^{s_2,l_2}$ and $\delta'>0$. Let $(u,n)\in B_{\delta'}^{C([0,T]; H^{s_1,l_1})}(0)$. Surely, 
	\begin{equation}\label{eq:appendix_continuity_Phi_estmate1}
	\norm{\Phi_{\icol{u_0\\ n_0}}(u,n)}_{L^\infty([0,T];H^{s_2,l_2})} \le \norm{\icol{e^{it\lap} u_0 \\ e^{-it\abs{\nabla}} n_0 }}_{L_t^\infty([0,T]; H^{s_2,l_2})}+\norm{N(u,n),(u,n)}_{L^\infty([0,T]; H^{s_2,l_2})}.
	\end{equation}
	We consider the two terms on the right side of the inequality \eqref{eq:appendix_continuity_Phi_estmate1} separately. Recall that both propagators are unitary. This yields \begin{equation*}
		\norm{\icol{e^{it\lap} u_0 \\ e^{-it\abs{\nabla}} n_0 }}_{L_t^\infty([0,T]; H^{s_2,l_2})} = \norm{\icol{u_0\\ n_0}}_{H^{s_2,l_2}}.
	\end{equation*} Now we split the second term on the right-hand side of \eqref{eq:appendix_continuity_Phi_estmate1}, i.e. $\norm{N(u,n),(u,n)}_{L^\infty([0,T]; H^{s_2,l_2})}$, into
	\begin{equation*}
		\underbrace{\norm{-i\int_0^t e^{i(t-s)\lap} u(s)\Re(n(s)) \di s }_{L_t^\infty([0,T];H^{s_2})}}_{(I)}+\underbrace{\norm{-i\int_0^t e^{-i(t-s)\abs{\nabla}} \abs{\nabla}u(s)\overline{u}(s) \di s }_{L_t^\infty([0,T];H^{l_2})}}_{(II)}.
	\end{equation*}
	To clarify the prerequisites of \cref{lem:sobolevproduct} for both (I) and (II) beforehand, notice $s_1,l_1\ge s_2$, $s_1+l_1= s_1+s_2>\tfrac{d}{2}+s_2$, $s_1\ge l_2+1$ and $2s_1 > l_2+1+\tfrac{d}{2}$. Triangle inequality and an application of \cref{lem:sobolevproduct} yield
	\begin{equation*}
	(I)\le \sup_{t\in[0,T]} t\norm{u(t)\Re(n(t))}_{H^{s_2}}\lesssim \sup_{t\in[0,T]} t\norm{u(t)}_{H^{s_1}}\norm{\Re(n(t))}_{H^{l_1}}\le T\norm{u}_{C([0,T];H^{s_1})} \norm{n}_{C([0,T];H^{l_1})}.
	\end{equation*}
	Similar arguments using \cref{lem:sobolevproduct} with $\wop{\cdot}^{l_2}\abs{\cdot}\le \wop{\cdot}^{l_2+1}$ yield the following bound for the second term $(II)$.
	\begin{align*}
	(II)\le  \sup_{t\in[0,T]} t\norm{ \abs{\nabla}u(t)\overline{u(t)}}_{H^{l_2}} \le  \sup_{t\in[0,T]} t\norm{u(t)\overline{u(t)}}_{H^{l_2+1}}\lesssim \sup_{t\in[0,T]} t \norm{u(t)}_{H^{s_1}} \norm{\overline{u(t)}}_{H^{s_1}}\le T \norm{u}_{C([0,T];H^{s_1})}^2.
	\end{align*}
	The continuous time dependence is obvious. Fix a small $\epsilon>0$ and $0<\delta\ll \tfrac{\epsilon}{4\delta' T}$. Let $(u_1,n_1), (u_2,n_2)\in B_{\delta'}^{C([0,T]; H^{s_1,l_1})}(0)$ with $\norm{(u_1,n_1)-(u_2,n_2)}_{C([0,T]; H^{s_1,l_1})}<\delta$. We estimate the difference of the images under the map $\Phi_{\icol{u_0\\n_0} }$ in $C([0,T];H^{s_2,l_2})$, which we transform with a dimensional constant 
	\begin{align*}
	&\norm{\Phi_{\icol{u_0\\ n_0}}(u_1,n_1) - \Phi_{\icol{u_0\\ n_0}}(u_2,n_2) }_{C([0,T]; H^{s_2}\times H^{l_2})}\\
	&\qquad\qquad\le\underbrace{\norm{-i\int_{0}^{t}e^{i(t-s)\lap} (u_1(s)\Re(n_1(s)) - u_2(s)\Re(n_2(s))) \di s }_{C([0,T];H^{s_2})}}_{(III)} \\
	&\qquad\qquad\quad+ \underbrace{\norm{-i\int_{0}^{t}e^{-i(t-s)\abs{\nabla}} \abs{\nabla} (\abs{u_1(s)}^2 - \abs{u_2(s)}^2) \di s}_{C([0,T];H^{l_2})}}_{(IV)}.
	\end{align*}
	We estimate both terms separately. 
	\begin{align*}
	(III)&\le T \sup_{t\in[0,T]} \norm{u_1 \Re(n_1)-u_2 \Re(n_2)}_{H^{s_2}}\\
	&\le  T \sup_{t\in[0,T]} \norm{u_1(t) (\Re(n_1(t))-\Re(n_2(t)))}_{H^{s_2}} +T \sup_{t\in[0,T]} \norm{(u_1(t)-u_2(t)) \Re(n_2(t))}_{H^{s_2}}.\\
	\intertext{Furthermore, the Sobolev product estimates \cref{lem:sobolevproduct} yield}
	(III)&\le  T \sup_{t\in[0,T]} \norm{u_1(t)}_{H^{s_1}} \norm{\Re(n_1)-\Re(n_2)}_{H^{l_1}} +T \sup_{t\in[0,T]} \norm{u_1-u_2}_{H^{s_1}}\norm{ \Re(n_2)}_{H^{l_1}}\\
	&\le T \delta'\norm{n_1-n_2}_{C([0,T];H^{l_1})}+T\delta' \norm{u_1-u_2}_{C([0,T];H^{s_1})}\le 2\, T \delta' \delta<\epsilon.
	\end{align*}
	Similarly, we estimate
	\begin{equation*}
		(IV)\le T \sup_{t\in[0,T]} \norm{\abs{\nabla}\left( \abs{u_1(t)}^2-\abs{u_2(t)}^2\right)}_{H^{l_2}}\le T \sup_{t\in[0,T]} \norm{ \abs{u_1(t)-u_2(t)}(\abs{u_1(t)}+\abs{u_2(t)})}_{H^{l_2+1}}.
	\end{equation*}
	The second inequality in the previous line follows from $\wop{\xi}^{l_2}\abs{\xi}\le \wop{\xi}^{l_2+1}\le \wop{\xi}^{s_1}$ for $\xi \in \BR^d$. We use \cref{lem:sobolevproduct} to estimate $(IV)$ further.
	\begin{equation*}
	(IV)\le T \sup_{t\in[0,T]} \norm{ u_1(t)-u_2(t)}_{H^{s_1}}\norm{\abs{u_1(t)}+\abs{u_2(t)}}_{H^{s_1}}\le 2\,T \delta'\norm{ u_1(t)-u_2(t)}_{C([0,T];H^{s_1})}\le 2\, T \delta' \delta<\epsilon.
	\end{equation*}
	Therefore, we conclude the continuity by
	\begin{equation*}
	\norm{\Phi_{\icol{u_0\\ n_0}}(u_1,n_1) - \Phi_{\icol{u_0\\ n_0}}(u_2,n_2) }_{L^\infty([0,T];H^{s_2,l_2})} <4\delta' \delta T\ll \epsilon.
	\end{equation*}
\end{proof}

\medskip

\printbibliography 
\medskip 
\small
E-mail address: \href{mailto:fgrube@math.uni-bielefeld.de}{fgrube@math.uni-bielefeld.de} 
\end{document}